\documentclass{amsart}
\usepackage{amsmath,amsfonts}
  \usepackage{graphics} 
  \usepackage{latexsym,amssymb,graphicx,hyperref}
  \usepackage{epsfig} 
  \DeclareGraphicsExtensions{.ps}

  \usepackage{hyperref} 

\newtheorem{thm}{Theorem}

\newtheorem{lemma}{Lemma}
\newtheorem{propn}{Proposition}

\theoremstyle{definition}
\newtheorem*{defn}{Definition}
\newtheorem*{remark}{Remark}

\newcommand{\be}{\begin{equation}}
\newcommand{\ee}{\end{equation}}
\def\Spin{\mathop{\rm Spin}\nolimits}
\def\Char{\mathop{\mathfrak{Char}}\nolimits}
\newcommand{\C}{\mathbb{C}}
\newcommand{\Z}{\mathbb{Z}}
\newcommand{\R}{\mathbb{R}}
\newcommand{\g}{\mathfrak{g}}
\newcommand{\h}{\mathfrak{h}}

\newcommand{\sov}{\mathfrak{so}(V)}
\newcommand{\sotv}{\widetilde{\so}(V)}
\newcommand{\so}{\mathfrak{so}}
\newcommand{\spp}{\mathfrak{sp}}
\newcommand{\sot}{\widetilde{\mathfrak{so}}}
\renewcommand{\sl}{\mathfrak{sl}}
\newcommand{\p}{\mathfrak{p}}
\newcommand{\dd}{\mathfrak{d}}
\newcommand{\gt}{{\widetilde{\g}}}
\newcommand{\dt}{{\widetilde{d}}}
\renewcommand{\dh}{{\hat{d}}}
\newcommand{\hti}{{\widetilde{\h}}}
\newcommand{\Lt}{\widetilde{L}}
\newcommand{\Dt}{\widetilde{D}}
\newcommand{\gh}{{\widehat{\g}}}
\newcommand{\hh}{{\hat{\h}}}
\newcommand{\Vh}{{\hat{V}}}

\newcommand{\rhoh}{{\hat{\rho}}}
\newcommand{\muh}{{\hat{\mu}}}

\newcommand{\Ah}{{\hat{A}}}
\newcommand{\rhot}{{\tilde{\rho}}}
\newcommand{\thetat}{{\tilde{\theta}}}
\newcommand{\gb}{\breve{\g}}
\newcommand{\Rb}{\breve{R}}
\newcommand{\rhob}{{\breve{\rho}}}
\newcommand{\mut}{{\tilde{\mu}}}
\newcommand{\sigt}{\tilde{\sigma}}

\newcommand{\Rt}{\tilde{R}}

\newcommand{\Pc}{{\mathcal P}}

\newcommand{\PR}{\Pc_R}
\newcommand{\W}{{\mathcal W}}

\newcommand{\sm}{\mathfrak{s}}

\newcommand{\Or}{$\mathcal{I}_{O }\ $}
\newcommand{\OR}{$\mathcal{I}_{R }\ $}
\renewcommand{\O}{\mathcal{O}}
\newcommand{\Ow}{\mathcal{O}_{\text{weak}}}
\newcommand{\aff}{\mbox{$\mathcal{I}_{A}$}}
\newcommand{\sh}[1]{{#1}}
\newcommand{\Xt}{\widetilde{X}}
\newcommand{\Zt}{\widetilde{Z}}
\newcommand{\Ft}{\widetilde{F}}

\newcommand{\Zh}{\hat{Z}}
\newcommand{\Fh}{\hat{F}}

\newcommand{\leavethisout}[1] {}

\begin{document}
\title[Tensor factorization]
      {Tensor factorization and Spin construction for Kac-Moody algebras}

\author{Rajeev Walia}


 \email{rwalia@math.ucr.edu}
 \address{Department of Mathematics, Michigan State University,
   East Lansing, MI, 48824, USA}


\begin{abstract}
In this paper we discuss the ``Factorization phenomenon" which
occurs when a representation of a Lie algebra is restricted to a
subalgebra, and the result factors into a tensor product of smaller
representations of the subalgebra. We analyze this phenomenon for
symmetrizable Kac-Moody algebras (including finite-dimensional,
semi-simple Lie algebras). We present a few factorization results
for a general embedding of a symmetrizable Kac-Moody algebra into
another and provide an algebraic explanation for such a phenomenon
using Spin construction. We also give some application of these
results for semi-simple finite dimensional Lie algebras.

We extend the notion of Spin functor from finite-dimensional to
symmetrizable Kac-Moody algebras, which requires a very delicate
treatment. We introduce a certain category of orthogonal
$\g$-representations for which, surprisingly, the Spin functor gives
a $\g$-representation in Bernstein-Gelfand-Gelfand category $\O$.
Also, for an integrable representation $\Spin$ produces an
integrable representation. We give the formula for the character of
Spin representation for the above category and work out the
factorization results for an embedding of a finite dimensional
semi-simple Lie algebra into its untwisted affine Lie algebra.
Finally, we discuss classification of those representations for
which $\Spin$ is irreducible.
\end{abstract}
\maketitle

\section*{Introduction}
The {\it factorization phenomenon} occurs when a representation of a
Lie algebra $\gt$ is restricted to a subalgebra $\g$, and the result
factors into a tensor product of $\g$-representations:
$$V\sh\downarrow^\gt_\g\,\cong\, V_1\otimes V_2\otimes\cdots\;\;\;.$$
We will consider general embeddings of symmetrizable Kac-Moody algebras $\g\subset\gt$.

This phenomenon has been widely studied when $\gt$ is an affine Lie
algebra and $\g$ its underlying finite-dimensional subalgebra; see
\cite{FL1},\cite{FL2},\cite{KMN},\cite{HK},\cite{Ka} and \cite{OSS}.
In this case, Fourier and Littelman \cite{FL1} have shown that every
irreducible $\gt$-representation factors into a tensor product of
infinitely many $\g$-representations.  Their proof by character
computations is essentially combinatorial.  Our work aims toward an
algebraic framework in which factorization appears functorially and
in a more general context, treating finite and infinite dimensional
Lie algebras simultaneously.

We define a large class of representations which exhibit tensor
factorization. First we give some motivation for this class in terms
of the charaters of its representations. For now, we consider
embedding of one semi-simple finite dimensional Lie algebra into
another, but we will see later that the arguments also work for
embeddings of symmetrizable Kac-Moody algebras with some additional
structure. We fix some notation:
\begin{itemize}
\item $\g\subset\gt$, an embedding of semi-simple finite dimensional Lie algebras\,.
\item $\rho$ = half sum of all positive roots of $\g$\,.
\item $\W$ = the Weyl group of $\g$.
\item $A_\mu\;=\;\sum_{w\in\W}\text{sign}(w)e^{w(\mu)}$, the skew symmetrizer of $e^{\mu}$ with respect to $\W$.
\item $R^+$\;=\; the set of all positive roots of $\g$.
\item $V(\lambda) = $irreducible representation with highest weight $\lambda$\,.
\item $V{\downarrow}_\g^{\gt}=$ restriction of a $\gt$-representation $V$ to $\g$.
\end{itemize}
Also, for an object $a$ associated to $\g$, we use $\tilde{a}$ to
denote the corresponding object for $\gt$. For example, $\rhot$
denotes the half sum of all positive roots of $\gt$ and $V(\rhot)$
denotes the irreducible representation of $\gt$ with highest weight
$\rhot$.

Consider the character of $V(\rhot)$. By Weyl denominator identity:
$$
A_\rhot=e^\rhot\prod_{\alpha\in\Rt^+}(1-e^{-\alpha})
$$
and Weyl character formula:
$$
\Char V(\lambda)=\frac{A_{\lambda+\rhot}}{A_\rhot},
$$
we obtain:
$$
\Char V(\rhot)= e^\rhot\prod_{\alpha\in\Rt^+}(1+e^{-\alpha}).
$$
This multiplicative form of the character of $V(\rhot)$ suggests
that the $\gt$-representation $V(\rhot)$ when restricted to $\g$,
might factor into tensor product of $\g$-representations.

Now, without loss of generality we may assume that Cartan subalgebra
of $\g$ is contained in Cartan subalgebra of $\gt$ and positive
roots of $\gt$ restrict to positive roots of $\g$. Then the
restriction of the $\gt$-character $\Char V(\rhot)$ to $\g$ will in
fact factor as follows:
$$\begin{array}{cccc}
\Char(V(\rhot){\downarrow}^\gt_\g)&=&\left(e^\rho\prod\limits_{\alpha\in R^+}(1+e^{-\alpha\downarrow})\right)&\left(e^{(\rhot\downarrow-\rho)}\prod\limits_{\alpha\in\Rt^+\setminus R^+}(1+e^{-\alpha\downarrow})\right)\\
&=&\Char V(\rho)&\left(e^{(\rhot\downarrow-\rho)}\prod\limits_{\alpha\in\Rt^+\setminus R^+}(1+e^{-\alpha\downarrow})\right)
\end{array}$$
where $\downarrow$ denotes restriction from $\gt$ to $\g$ and $R^+$ is
any subset of $\Rt^+$ which on restriction to $\g$ forms the set of
all positive roots of $\g$. Now, if we can find a
$\g$-representation whose character is the second factor above, we
can conclude that for any embedding $\g\subset\gt$ of semi-simple
Lie algebras, the irreducible $\gt$-representation $V(\rhot)$, when
restricted to $\g$, always factors into a tensor product of at least
two $\g$-representations, one of them being $V(\rho)$.

The $\g$-representation whose character is the second factor above is
obtained using Panyushev's \cite{P} reduced Spin functor $\Spin_0$. We
will define $\Spin_0$ in \S\ref{Spin} (and briefly in \S\ref{Props}).
Basically, $\Spin$ for a given Lie-algebra $\g$ is a functor from the
category of all $\g$-representations which have a non-degenerate symmetric
bilinear form preserved by the action of $\g$ (called {\it orthogonal}
$\g$-representations) to the category of all $\g$-representations. By
reducing multiplicities in the resulting representation we obtain
$\Spin_0$ which has the remarkable property that:
$$
\Spin_0(V_1\oplus V_2)\cong\Spin_0(V_1)\otimes\Spin_0(V_2).
$$

It is a well known fact that for any semi-simple finite dimensional
Lie algebra $\g$, the representation $V(\rho)$ can be realized as
$\Spin_0$ of adjoint representation of $\g$ (which is indeed
orthogonal due to the invariant Killing form). Thus, for the Lie
algebra $\gt$,
$$
V(\rhot)=\Spin_0(\gt).
$$
When we restrict to $\g$, it turns out that $\Spin_0$ commutes with
the restriction according to : $$\Spin_0(\gt){\downarrow}\;\cong
2^r\Spin_0(\gt{\downarrow}),$$ where $r$ is the number of positive
roots of $\gt$ which restrict to zero. Now, $\gt\cong
\g\oplus\g^\bot$ as $\g$-representation, where $\bot$ denotes the
orthogonal complement with respect to the Killing form. Therefore,
$$
\begin{array}{ccl}
    V(\rhot){\downarrow}^\gt_\g&\cong& 2^r\Spin_0(\g\oplus\g^{\bot})\\
                    &\cong& 2^r[\Spin_0(\g)\;\;\otimes\;\;\Spin_0(\g^\bot)],
\end{array}
$$
by the property of $\Spin_0$ mentioned above. So,
$$
\begin{array}{ccl}
    V(\rhot){\downarrow}^\gt_\g&\cong& \Spin_0(\g)\;\;\otimes\;\;[2^r\Spin_0(\g^\bot)]\\
                    &\cong& V(\rho)\;\;\otimes\;\;[2^r\Spin_0(\g^\bot)],
\end{array}
$$
as $V(\rho)$ is isomorphic to $\Spin_0(\g)$. Hence, we get the
tensor factorization of the restricted $V(\rhot)$. This is the
content of Theorem \ref{t1} in \S\ref{FT} where it is extended to embedding
of symmetrizable Kac-Moody algebras with some additional structure.
The detailed proof is given later.

From this, using Weyl character formula, we can obtain a tensor
factorization of the $\gt$-representation $V(2\mut+\rhot)$ for {\it
any} dominant weight $\mut$, which forms the content of Theorem \ref{t2} in
\S\ref{FT}. In \S\ref{Props}, we state some important properties of
the $\Spin$ functor. We describe some consequences of the above
theorems for finite dimensional semi-simple Lie algebras and
untwisted affine Lie algebras in \S\ref{SpeCases} and \S\ref{Affine}
respectively. In \S\ref{CoP}, for a subclass of orthogonal
representations of untwisted affine Lie algebras, called affinized
representations, we classify those whose $\Spin_0$ is irreducible.
Following Panyushev \cite{P} these are called coprimary
representations.
\section{Main results}\label{MR}
\subsection{Background for symmetrizable Kac-Moody algebras}\label{Bkground}
An $n\times n$ matrix, $A=(a_{ij})$, is called a {\it generalized Cartan matrix} if:
\begin{enumerate}
\item $a_{ii}=2$ for all $i=1,2,\cdots,n.$
\item $a_{ij}$ is a non-positive integer for all $i\neq j$.
\item $a_{ij}=0$ implies $a_{ji}=0$ for all $i\neq j$.
\end{enumerate}
For any $n\times n$ matrix $A=(a_{ij})$ of rank $l$, we define {\it a realization of $A$} as a triple $(\h,\Pi,\Pi^\vee)$, where $\h$ is a complex vector space, $\Pi=\left\{\alpha_1,\alpha_2,\cdots,\alpha_n\right\}\subset\h^*$ and $\Pi^\vee=\left\{\alpha_1^\vee,\alpha_2^\vee,\cdots,\alpha_n^\vee\right\}\subset\h$ are indexed subsets in $\h^*$ and $\h$ respectively, satisfying the following three conditions:
\begin{enumerate}
\item Both sets $\Pi$ and $\Pi^\vee$ are linearly independent.
\item $\alpha_j(\alpha_i^\vee)=a_{ij}$ for all $i,j=1,2,\cdots,n$.
\item dim$(\h)=2n-l$.
\end{enumerate}
Two realizations $(\h,\Pi,\Pi^\vee)$ and $(\h_1,\Pi_1,\Pi_1^\vee)$ are called isomorphic if there exists a vector space isomophism $\phi:\h\longrightarrow\h_1$ such that $\phi(\Pi^\vee)=\Pi_1^\vee$ and $\phi^*(\Pi_1)=\Pi$. There exists a unique (up to isomorphism) realization of every $n\times n$ matrix. The realizations of two matrices $A$ and $B$ are isomorphic if $B$ can be obtained from $A$ by a permutution of the indexing set \cite[Proposition 1.1]{Ka}.

An $n\times n$ matrix $A$ is called {\it symmetrizable} if there exists an invertible diagonal matrix $D=\text{diag}(\epsilon_1,\epsilon_2,\cdots,\epsilon_n)$ and a symmetric matrix $B$ such that $A=DB$.

\begin{defn} ({\bf Symmetrizable Kac-Moody algebra})\; Let $A=(a_{ij})$ be a symmetrizable generalized Cartan matrix and let $(\h,\Pi,\Pi^\vee)$ be a realization of $A$. A symmetrizable Kac-Moody algebra, $\g$, associated to $A$ is defined as the Lie algebra on generators $X_{\pm i}$ ($i=1,\cdots,n$), all $H\in\h$ with the following defining relations:
\begin{enumerate}
\item $[H_1,H_2]=0$ \hspace{66pt}for all $H_1, H_2\in\h$.
\item $[H,X_{\pm i}]=\pm\alpha_i(H)X_{\pm i}$ \hspace{18pt}for all $i=1,\cdots,n$ and $H\in\h.$
\item $[X_i,X_{-j}]=\delta_{ij}\alpha_i^\vee$\hspace{47pt}for all $i,j=1,\cdots,n.$
\item $\text{ad}(X_{\pm i})^{1-a_{ij}}(X_{\pm j})=0$\hspace{22pt}for all $i, j=1,\cdots,n.$
\end{enumerate}
Here, $\text{ad}(X)(\cdot):=[X,\;\cdot\;]$ and $\h$ is called the Cartan subalgebra of $\g$.
\end{defn}
Let $\g$ be symmetrizable Kac-Moody algebra. We define a
non-degenerate symmetric bilinear form $(\;\cdot\;,\;\cdot\;)$ on
$\h$ which can be extended (See \cite[Thm 2.2]{Ka}) to a
non-degenerate symmetric bilinear form on whole of $\g$ such that
$(\;\cdot\;,\;\cdot\;)$ is preserved by the adjoint action of $\g$,
that is:
$$([X,Y],Z)+(Y,[X,Z])=0$$
for all $X,Y,Z\in\g$. Let $A$ be a symmetrizable generalized Cartan matrix with a fixed decomposition $A=DB$ (See the definition of a symmetrizable matrix above). Let $\h^{'}:=\bigoplus_{i=1}^n\C\alpha_i^\vee$. Fix a complementary space $\h^{''}$ to $\h^{'}$ in $\h$ and define:
$$(\alpha_i^\vee,H)=\alpha_i(H)\epsilon_i\;\;\;\;\forall\;\;H\in\h;$$
$$(H_1,H_2)=0\;\;\;\;\forall\;\;H_1,H_2\in\h^{''}.$$

\subsection{Factorization Theorems}\label{FT}

We now state our main factorization results using $\Spin$
construction. We consider embeddings, $\g\subset\gt$, of
symmetrizable Kac-Moody algebras. Our analysis deals with finite as
well as infinite dimensional representations.  For example, we
consider infinite dimensional irreducible representations of an
affine Lie algebra $\gt$ with finite dimensional weight spaces. If
we restrict such a representation to a finite dimensional Lie
algebra $\g$, the $\g$-weight spaces no longer remain finite
dimensional. To avoid this, we enlarge the Cartan subalgebra of $\g$ by 
one dimension by augmenting an element $d$ from Cartan subalgebra of $\gt$
so that for certain class of $\gt$-representations, the restricted
representation to ${\g}{\oplus}{\C d}$ has finite-dimensional weight spaces. 
This gives rise to the following notions: an {\it
augmented symmetrizable Kac-Moody algebra}, a {\it $d$-embedding} of
such algebras, say $\g\subset\gt$, and {\it $d$-finite} representations so that any
$d$-finite $\gt$-representation, when restricted to $\g$, has
finite dimensional weight spaces.

\begin{defn} {(\bf Augmented symmetrizable Kac-Moody algebra $\g$)} A Lie algebra $\g$ is called an augmented symmetrizable Kac-Moody algebra if $\g$ has a certain distinguished element $d$ in the Cartan subalgebra $\h$ of $\g$ such that either:
\begin{itemize}
\item  $\g$ itself is a symmetrizable Kac-Moody algebra and $\alpha_i(d)\in\Z_{>0}$ for all  $\alpha_i\in\Pi$,
\item  or $\g=\g_1\oplus\C d$ where:
\begin{itemize}
\item $\g_1$ is a symmetrizable Kac-Moody algebra.
\item for each root vector $X_{\pm\alpha}$ of $\g_1$,
$[d,X_{\pm\alpha}]=\pm c_\alpha X_{\pm\alpha}$, for some {\it positive} integer $c_\alpha$. 
\end{itemize}
\end{itemize}
\end{defn}

\begin{remark}When $\g=\g_1\oplus\C d$, we can extend the action of the root $\alpha$ of $\g_1$ to $(\h_1\oplus\C d)$ by defining $\alpha(d):=c_\alpha$. Thus $\h:=\h_1\oplus\C d$ is the Cartan subalgebra of $\g$.
\end{remark}
Example: Let $\g_1\cong\sl_2\C:=\C H_\alpha\oplus\C X_\alpha\oplus\C X_{-\alpha}$ with usual bracket relations. Define $[d,H_\alpha]:=0$ and $[d,X_{\pm\alpha}]:=\pm X_{\pm\alpha}$ , so that $\alpha(d):=1$. Then, $\g:=\g_1\oplus\C d$ is an augmented symmetrizable Lie algebra.

\begin{defn} {\bf ($d$-embedding)}\;An embedding $\g\subset\gt$ of augmented symmetrizable Kac-Moody algebras with distinguished element $d$ and $\dt$ and Cartan subalgebras $\h$ and $\hti$ respectively, is called a $d$-embedding
if\;:
\begin{itemize}
\item $d=\dt$,
\item $\h\subset\hti$ and
\item positive roots of $\g$ are restrictions of positive roots of $\gt$.
\end{itemize}
\end{defn}
Let $\g$ be an augmented symmetrizable Kac-Moody algebra with distinguished element $d$ in the Cartan subalgebra $\h$ and weight lattice $\Pc$. For a $\g$-representation $V$ and $\Lambda\in\Pc$, let $V^{(\Lambda)}:=\left\{v\in V:H(v)=\Lambda(H)v \hspace{.5em}\forall\hspace{.5em} H\in\h\right\}$ denote the corresponding weight space of $V$. $\Lambda$ is called a {\it weight of $V$} if $V^{(\Lambda)}\neq\left\{0\right\}$.

For the distinguished element $d$ in the Cartan subalgebra of $\g$, we say a $\g$-representation $V$ is {\it $d$-finite} if\;\;:
\begin{itemize}
\item $\Lambda(d)\in\Z\setminus\{0\}$ for all non-zero weights $\Lambda$ of $V$ and
\item $\bigoplus_{\Lambda(d)=k}V^{(\Lambda)}$ is finite-dimensional for each $k\in\Z$\;.
\end{itemize}

As we mentioned in the Introduction, the input for the $\Spin$
functor is an orthogonal representation which we defined for
semi-simple finite dimensional Lie algebras. The same definition
extends to the augmented symmetrizable Kac-Moody algebra too. For an
augmented symmetrizable Kac-Moody algebra $\g$, a
$\g$-representation $V$ is called {\it orthogonal} if there exists a
non-degenerate symmetric bilinear form $Q$ on $V$, invariant under
the action of $\g$, that is, $Q(Xu,v)+Q(u,Xv)=0$ for all $u,v\;\in
V$ and $X\in\g$. For example, the action of a symmetrizable
Kac-Moody algebra on itself by brackets, called adjoint
representation,\;{\it is} orthogonal due to the invariant bilinear
form (see \S\ref{Bkground}).

Next we define the adjoint representation of an augmented symmetrizable Kac-Moody algebra in such a way that it is orthogonal, so that we can apply the $\Spin$ functor to it (see the Introduction). Adjoint representation for an augmented symmetrizable Kac-Moody algebra $\g$ with distinguished element $d$ is already defined if $\g$ is itself a symmetrizable Kac-Moody algebra. So, let $\g=\g_1\oplus\C d$. In this case, the action of $\g$ on $\g_1$ by brackets is defined as the {\it adjoint representation} of $\g$. We can show that the action of the distinguished element $d$ preserves the bilinear form on $\g_1$ and thus $\g_1$ is orthogonal as a $\g$-representation.
\begin{remark} It is easy to show that for a $d$-embedding, $\g\subset\gt$, the adjoint representation of $\gt$ is $d$-finite and orthogonal both as a $\gt$-representation and a $\g$-representation.
\end{remark}
Theorems \ref{t1} and \ref{t2} (given below) describe a class of representations which exhibit the factorization phenomenon (see the Introduction). We will follow the notations used in the Introduction except that $\g\subset\gt$ will denote a $d$-embedding of two augmented symmetrizable Kac-Moody algebras and $\rho$ and $\rhot$ will denote the sum of all fundamental weights of $\g$ and $\gt$ respectively.

\begin{thm} \label{t1}
{\it For a $d$-embedding, $\g\subset\gt$, of augmented symmetrizable Kac-Moody algebras, suppose the adjoint representation of $\gt$ decomposes into orthogonal $\g$-representations as: $\gt\cong\g\oplus\p_1\oplus\p_2\oplus\cdots$\,.
Then the $\gt$-representation  $V(\rhot)$, when restricted from $\gt$ to $\g$,
factors into a tensor product of $\g$-representations as:
$$V(\rhot){\downarrow}^\gt_\g\ \cong\, V(\rho)\otimes\;W_1\otimes W_2\otimes\cdots$$
with $W_j=\Spin_0(\p_j)$\,,
where $\Spin_0$ is the reduced Spin functor defined in \S\ref{Spin}.
}
\end{thm}
In the finite-dimensional case, the Theorem is closely related to the results of Kostant \cite{K1,K2}.

Theorem \ref{t1} leads to  a large class of representations  exhibiting tensor factorization.

\begin{thm} \label{t2}
{\it Let $\g\subset\gt$ be as in Theorem \ref{t1}, and let $\mut$ be a dominant weight of $\gt$. Then the $\gt$ representation $V(2\mut+\rhot)$, when restricted to $\g$, factors into a tensor product of  $\g$-representations which include the same $W_j$ as in Theorem \ref{t1}.
The other factor can be expressed in terms of the irreducible decomposition of the restricted $V(\mut)$.

That is, if we let:
$$V(\mut){\downarrow}^\gt_\g \,\cong\,\bigoplus_{i}V(\mu_{i})\,,$$

\noindent then:
$$V(2\mut{+}\rhot){\downarrow}^\gt_\g\,\cong\,\left(\bigoplus_{i}V(2\mu_{i}{+}\rho)\right)\otimes\;
W_1\otimes W_2\otimes\cdots\;.
$$
}
\end{thm}
We will prove these theorems in \S\ref{Proofs}.
\subsection{Basic properties of the $\Spin$ functor}\label{Props}
We now describe the basic properties of $\Spin$ functor, reserving the more technical discussion for \S\ref{Spin}. In the finite-dimensional case the construction is quite simple, and was examined by Panyushev \cite{P}. For an $n$-dimensional vector space $V$ with a non-degenerate symmetric bilinear form, recall that the orthogonal Lie algebra $\sov$ has a representation on the $2^{\lfloor n/2\rfloor}$-dimensional space $\Spin(V):=\wedge^\bullet V^+$\,, the total wedge space or exterior algebra of a maximal isotropic subspace $V^+\subset V$.

Let $\g$ be a semi-simple, finite dimensional Lie algebra. For an orthogonal $\g$-representation $V$, $\g$ acts by orthogonal matrices: that is, through an embedding $\g\subset\sov$\,. Restricting the action of $\sov$ makes $\Spin(V)$ a representation of $\g$.
If the zero weight space of $V$ has dimension $r$, it turns out that $\Spin(V)$ can be decomposed as the direct sum of $2^{\left\lfloor r/2\right\rfloor}$ copies of a smaller representation, which we call $\Spin_0(V)$.

Now let $\g$ be an augmented symmetrizable Kac-Moody algebra. In \S\ref{Spin}, we will define the $\g$-representation $\Spin(V)=\wedge^\bullet V^+$ for $V$ in the category of all $d$-finite and orthogonal (possibly infinite-dimensional) $\g$-representations. We will prove that the output, $\Spin(V)$, will be a $d$-finite $\g$-representation in the category $\Ow$ (defined below). Category $\Ow$ contains the Bernstein-Gelfand-Gelfand category $\O$ and has similar properties. Further if the input representation $V$ is root finite (defined below) then we will prove, $\Spin(V)$ belongs to $\O$.  If the zero weight space of $V$ is even, $\Spin(V)$ decomposes into direct sum of  $\g$-representation which we call {\it half-Spin representations} $\wedge^{\text{even}} V^+$ and $\wedge^{\text{odd}} V^+$. We also denote these by $\Spin^{\text{even}}(V)$ and $\Spin^{\text{odd}}(V)$.

We define a partial ordering\;$\leq$ called root order, on the the weight lattice $\Pc$ of $\g$ as follows:
We say, $\beta\leq\gamma$ in the {\bf root order} if $\gamma-\beta=\sum_{\alpha}c_\alpha\alpha$ where $\alpha$ is a simple positive root of $\g$ and $c_\alpha\in\Z_{\geq 0}$ for all $\alpha$.

For an augmented symmetrizable Kac-Moody algebra $\g$ and a
$\g$-representation $V$ define:
$$M_V:=\text{Set of all weights of } V \text{\;maximal in the root order}.$$
\begin{defn} {\bf (Category $\Ow$ of $\g$-representations)}\; $\Ow$ consists of all $\g$-representations $V$ such that for each weight $\beta$ of $V$:
\begin{enumerate}
\item  the weight space $V^{(\beta)}$ is finite dimensional and
\item  there exists $\lambda\in M_V$ such that $\beta\leq\lambda$\;.
\end{enumerate}
\end{defn}

\begin{remark}
 The morphisms in $\Ow$ are $\g$-representations homomorphisms. Following fact can be deduced, using \cite[Thm 10.7]{K}, that for a representation $V$ in $\Ow$\;which is integrable (meaning simple root vectors act locally nilpotently), the isomorphism class of $V$ is determined by its character.
\end{remark}
The well-known {\it Bernstein-Gelfand-Gelfand category $\O$} of $\g$-representations can be defined as a subcategory of $\Ow$:
$$\O:=\left\{V\in\Ow\;:\; M_V \text{\;is a finite set}\right\}\;.$$
It is worth noting that $\O$ can be defined to consist of $\g$-representations $V$ such that $V$ has finite-dimensional weight spaces and there exists a finite set $F$, a subset of weight lattice $\Pc$ of $\g$, so that for each weight $\beta$ of $V$ we can find $\lambda\in F$ with $\beta\leq\lambda$.

Let $\left\{\alpha_i\right\}_{i=1}^n$ denote the simple positive
roots of an augmented symmetrizable Kac-Moody algebra $\g$ with
distinguished element $d$. Let $\h^*$ be the dual Cartan subalgebra
of $\g$. Define the root cone:
\[ C:=\left\{\sum_{i=1}^n a_i\alpha_i\in\h^*:a_i\in\mathbb{R}_{\geq 0}\;\forall\; i\;
\text{or}\; a_i\in\mathbb{R}_{\leq 0}\;\forall\; i\right\}. \]

\begin{defn} {\bf Root finite $\g$-representation}\; We say a
$\g$-representation $V$ is root-finite if for every weight $\Lambda$
of $V$, $V^{(\Lambda)}$ is finite dimensional,
$\Lambda(d)\in\Z\setminus\left\{0\right\}$ for $\Lambda\neq 0$ and there are
only finitely many weights of $V$ in $\h^*\setminus\;C$.
\end{defn}

Propositions \ref{p1}  and \ref{p2} give some basic properties of $\Spin(V)$.
\begin{propn}\label{p1}
{\it Let $\g$ be an augmented symmetrizable Kac-Moody algebra with distinguished element $d$. Let $V$,\;$V_1$ and $V_2$ be $d$-finite and orthogonal $\g$-representations.
\begin{enumerate}
\item $\Spin(V)$ is $d$-finite and belongs to $\Ow$.
\item $V$ is integrable $\Rightarrow\;\Spin(V)$ is integrable and $\Spin(V)\cong W^{\oplus r}$ for some $\g$-representation $W$ called $\Spin_0(V)$. Here $r=\left\lfloor m_0/2\right\rfloor$ where $m_0$ is the dimension of the zero weight space, $V^{(0)}$ of $V$.
\item Let $m_i:=\text{dim}(V^{(0)}_i)$ for $i=1,2$.\\
If at least one of $m_1$ or $m_2$ is even, then:
$$\Spin(V_1\oplus V_2)\;\cong\;\Spin(V_1)\otimes\Spin(V_2).$$
If both $m_1$ and $m_2$ are odd, then:
$$\Spin^{\text{even}}(V_1\oplus V_2)\;\cong\;\Spin(V_1)\otimes\Spin(V_2)\cong\Spin^{\text{odd}}(V_1\oplus V_2).$$
and $$\Spin(V_1\oplus V_2)\;\cong\;(\;\Spin(V_1)\otimes\Spin(V_2)\;)^{\oplus 2}.$$
\item If $V_1$ and \;$V_2$ are integrable then
$$\Spin_0(V_1\oplus V_2)\;\cong\;\Spin_0(V_1)\otimes\Spin_0(V_2).$$
\item $W$ is root-finite $\Rightarrow$\; $W$ is $d$-finite.
\item $V \text{\;is root-finite}\Leftrightarrow\Spin(V)\in\O$
\item For adjoint representation $\g$,
$$\Spin_0(\g)\cong V(\rho)\,.$$
\end{enumerate}}
\end{propn}
Let \;$$\text{\Or} = \text{Category of all}\; d\text{-finite and orthogonal}\; \g\text{-representations}\;,$$
$$\text{\OR} =\text{\;Category of all root-finite and orthogonal\;} \g\text{-representations}.$$ Then, by Proposition \ref{p1}(5), $$\text{\OR}\subset\text{\Or}$$ and by Proposition \ref{p1}(1) and \ref{p1}(6), $\Spin(V)$ is a functor from the category \Or to the category $\Ow$ and also from category \OR to category $\O$. Thus,
$$
\begin{array}{ccl}
\text{\Or}&\stackrel{\Spin}{\longrightarrow}&\;\;\Ow\\
\cup\;\;      &                                 &\;\;\cup\\[-.2em]
\text{\OR}&\stackrel{\Spin}{\longrightarrow}&\;\;\O
\end{array}
$$
The following Proposition gives the character of $\Spin_0(V)$ in terms of the character of $V$.

\begin{propn} \label{p2}
{\it Let $V$ be an integrable $\g$-representation in \Or. Let $m_{\beta}$ be the multiplicity of a weight $\beta$ of $V$ so that the character of $V$ can be written as:
$$\begin{array} {rcl}
    \Char V&=&\sum\limits_{\beta(d)>0}m_{\beta}(e^{\beta}+e^{-\beta})\ +\ m_0\,.
\end{array}$$
Then the $\g$-representation $\Spin_0(V)$ has the character:
$$\Char\Spin_0(V)
\ =e^{\Lambda}
\mathop{\prod}_{\beta(d)>0} (1\sh+e^{-\beta})^{m_{\beta}}\,.
$$
Here $\Lambda:=\sum_{i=1}^n c_{i}\Lambda_i$, where
$\left\{\Lambda_i\right\}_{i=1}^n$ are the fundamental weights and
the coefficient $c_i$ is defined as follows:
$$c_i:=\sum\frac{1}{2}m_\beta\beta(\alpha_i^\vee),$$
where the sum is over all weights $\beta$ of $V$ such that $\beta(d)>0$ and $s_i(\beta)(d)<0$. Here $s_i$ denotes the reflection in the plane perpendicular to the simple root $\alpha_i$.
Because of the $d$-finiteness of $V$, $c_i$ has finitely many nonzero terms .}
\end{propn}
\begin{remark} When $V$ is finite dimensional, the $\Lambda$ simplifies to $\Lambda=\sum_{\beta(d)>0}\frac{1}{2}m_\beta\beta.$
\end{remark}

\subsection{Special cases for Finite-dimensional Lie algebras}\label{SpeCases}
We give some special cases of Theorems \ref{t1} and \ref{t2} when
$\g\subset\gt$ is an arbitrary embedding of finite-dimensional
semi-simple Lie algebras. This embedding can be turned into a
$d$-embedding of augmented symmetrizable Kac-Moody algebras, by
appropriately choosing a $d$ in the Cartan subalgebra of $\g$.

\subsubsection{Principal Specialization}
We let $\g\subset\gt$ be the embedding of a principal three-dimensional subalgebra in the special linear Lie algebra:
$
\sl_2(\C)\subset\sl_n(\C)\,
$, defined as  $\sl_2(C):=\C X\oplus\C Y\oplus\C H$, where
$$
\begin{array}{ccl}
X&:=&\sum\limits_{i=1}^{n-1}i\;E_{i,i+1},\\\\
Y&:=&\sum\limits_{i=1}^{n-1}(n-i)E_{i+1,i},\\\\
H&:=&\sum\limits_{i=1}^{n}(n+1-2i)E_{i,i}.
\end{array}
$$
Here, $E_{i,j}$ denotes the $n\times\;n$ matrix which has $1$ at $(i,j)^{\text{th}}$ place and zero elsewhere.

The character of a $\sl_n$-irreducible $V(\mu)$ is the Schur polynomial  $S_\mu(x_1,$ $\ldots,x_n)$\,.  Its restriction from $\gt$ to $\g$ corresponds to the {\it principal specialization} $x_i\sh\downarrow^\gt_\g=q^{i-1}$, where $q=e^\alpha$ for $\alpha$ the simple root of $\g$.
Theorem \ref{t2} implies the following factorization of the specialized Schur function:
\begin{propn}\label{p3}
{\it
$$\begin{array}{l}
S_{2\mu+\rho}(1,q,q^2,\cdots,q^{n-1})\\[.5em]
\qquad=\left(\, q^{\binom{n}{3}}(1\sh+q)\,S_\mu(1,q^2,q^4,\ldots,q^{2n-2}\,)\right)\cdot w_1(q)\cdot w_2(q)\cdots w_{n-2}(q)\,,
\end{array}$$
where $w_k(q)=(1\sh+q)(1\sh+q^2)\cdots(1\sh+q^{k+1})$\,,
$\rho=(n\sh-1,\ldots,1,0)$\, and all $n-1$ factors on the right-hand side of the formula are symmetric unimodal $q$-polynomials.}
\end{propn}

\begin{defn}{\bf(Symmetric unimodal polynomial)}\;A polynomial of the form, $f(q)=\sum_{i=N}^M a_iq^i$, is symmetric unimodal if  $a_{N+i}=a_{M-i}$ for all $i$, and $a_N\leq \cdots\leq a_K\geq a_{K+1}\geq \cdots \geq a_M$ for some
$K$.
\end{defn}
Proposition \ref{p3} is a kind of multiplicative analog of a result of Reiner and Stanton which states that
for certain pairs $\lambda,\mu$, the centered difference $S_{\lambda}(1,\ldots,q^{n-1})- q^N S_{\mu}(1,\ldots,q^{n-1})$ is symmetric unimodal.

\subsubsection{Folding of Dynkin diagrams}
Let $\gt$ be a simple Lie algebra with Dynkin diagram $\Dt$. A graph automorphism $\phi$ of $\Dt$, induces an automorphism, call it $\phi$ again, on $\gt$. We let $\g$ be the fixed subalgebra, under this automorphism, $\phi$. Then the Dynkin diagram, $D$ of $\g$ is called the folding of $\Dt$. For such an embedding, $\g\subset\gt$, Theorem \ref{t1} implies:
\begin{propn}\label{p4}
{\it
$$
V(\rhot)\sh\downarrow^{\gt}_{\g}\;\; \cong\;\; V(\rho)\;\;\otimes\;\; \left[\;V(e(\rho\sh+\rho_s)\sh+\rho_s)\,\oplus\,(a\sh-2)V(0)\;\right]^{\,\otimes\, a-1}\,,
$$
where $\rho_s$ is the half-sum of the positive {\it short} roots of $\g$\,, \ $a$ is the order of the automorphism $\phi$\,, and $e$ is the number of edges $\stackrel{i}\bullet
\stackrel{\!\line(1,0){19}}{}
\stackrel{j}\bullet$\; in $\widetilde{D}$ such that $\phi$ exchanges $i$ and $j$.}
\end{propn}
For example, the natural embedding  $\mathfrak{so}_{2n+1}\C\subset\mathfrak{sl}_{2n+1}\C$
corresponds to horizontally folding the diagram $A_{2n}$ to obtain $B_n$\,:
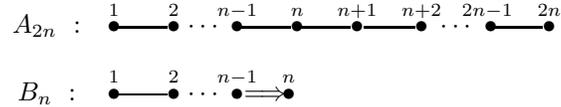
$$\begin{array}{cl}
A_{2n}\ :& \stackrel{1}\bullet
\stackrel{\!\line(1,0){19}}{}
\stackrel{2}\bullet
\cdots
\!\!\stackrel{n-1}\bullet\!\!\!
\stackrel{\!\line(1,0){19}}{}
\stackrel{n}\bullet
\stackrel{\!\line(1,0){19}}{}
\!\!\!\stackrel{n+1}\bullet\!\!
\stackrel{\!\line(1,0){19}}{}
\!\!\!\stackrel{n+2}\bullet\!\!
\cdots
\!\!\stackrel{\!\!2n-1}\bullet\!\!\!\!
\stackrel{\!\line(1,0){19}}{}
\!\stackrel{2n}\bullet
\\[1em]
B_n\ :& \stackrel{1}\bullet
\stackrel{\!\line(1,0){19}}{}
\stackrel{2}\bullet
\cdots
\!\!\stackrel{n-1}\bullet\!\!\!
\Longrightarrow
\!\stackrel{n}\bullet
\end{array}$$
The automorphism is $\phi(i)=2n\sh-i\sh+1$ of order $a=2$ with a single folded edge so that $e=1$\,.Thus, 
$V(\rhot)\sh\downarrow^{\gt}_{\g}\,\cong\, V(\rho)\otimes V(\rho\sh+2\rho_s)\,.$

\subsection{Factorization Theorems for affine Lie algebras}\label{Affine}
The most remarkable aspect of our construction appears when $\gt=\gh$ is the untwisted affine Lie algebra associated to a finite-dimensional semi-simple algebra $\g$:
$$
\gh=\g{\,\otimes\,} \C[t,t^{-1}]\,\oplus\,\C K\,\oplus\,\C d\,.
$$
Here $K$ is the central element and $d$ the canonical derivation.
We also let $\Lambda_0$ be the distinguished fundamental weight, and $\delta$ the minimal imaginary root.  Also, if $a$ is an object associated to $\g$\,, then $\hat{a}$ denotes the corresponding object for $\gh$.

Let $\dh:=\rho^\vee+hd$ where $h:=\sum_{i=0}^n a_i$ is the Coxeter number. Here $a_i$'s are the numeric labels of the Dynkin diagram of $\gh$ in \cite[Page 54]{K}) and $\rho^\vee$ is the sum of all fundamental co-weights of $\g$, that is $\alpha(\rho^\vee)=1$ for all simple positive roots $\alpha$ of $\g$. Then ${\g}{\oplus}{\C\dh}\subset\gh$ is a $\dh$-embedding of augmented symmetrizable Kac-Moody algebra.

Let $\Vh\in$\;\OR (See \S\ref{Props} for definition of \OR) be an integrable $\gh$-representation of level zero, that is the center $K$ acts by zero. Then, even though the input $\gh$-representation $\Vh\in$ \OR has level zero, by Proposition \ref{p2} the output $\Spin(\Vh)$
is a representation of {\it positive} level in the Bernstein-Gelfand-Gelfand category $\O$.  That is, $\Spin$ is a functor from the category of graded level zero representations in \OR, a sub-category of the graded level-zero representations \;$\mathcal{I}$\; examined by Chari and Greenstein \cite{CG}, to the positive-level category $\O$.

Now we introduce a subcategory\;$\aff$ of \OR. We will work out Theorems \ref{t1} and \ref{t2} for this sub-category.

\subsubsection{Affinized representations}\label{Aff-Rep}

For the remainder of this section we will work with
$\gh$-representations in a more restricted class \aff $\ \subset$
\OR\!, the subcategory of affinizations of finite-dimensional
orthogonal $\g$-representations. That is, for an orthogonal
$\g$-representation $V$, its affinization is the
$\gh$-representation $$\Vh:=\bigoplus_{k\in \Z}t^kV$$ where the loop
algebra acts as $t^lX\!\cdot t^kv:=t^{k+l}(X\cdot v)$ for
$X\,{\in}\,\g$\,, $v\,{\in}\, V$; the center acts as $0$\,; and the
derivation $d$ acts as $t\frac{d}{dt}$\,.  This inherits a
non-degenerate symmetric bilinear form from $V$. Choose a strictly
dominant co-weight $d_1$ in the Cartan subalgebra $\h$ of $\g$ such
that $\beta(d_1)\in\Z\setminus\left\{0\right\}$ for all weights $\beta$ of
$V$. Since $V$ is finite dimensional, for sufficiently large $N$,
$-N<\beta(d_1),\theta(d_1)< N$ for all weights $\beta$ of $V$ and
highest root $\theta$ of $\g$. Define $\dh:=Nd+d_1$. It can be
verified that the weights of $\Vh$ are of the form
$\Lambda:=k\delta+\beta$ for $k\in\Z$ and $\beta$ a weight of $V$.
Thus for all non-zero weights $\Lambda$ of $\Vh$,
$\Lambda(\dh)\in\Z\setminus\left\{0\right\}$. Now, ${\g}{\oplus}{\C\dh}\subset\gh$
is a $\dh$-embedding of augmented symmetrizable Kac-Moody algebras.
Also, it is easy to check that $\Vh\in$ \OR and is an integrable
$\gh$-representation.

For the representations $\Vh\in\aff$, we refine Proposition \ref{p2} below to obtain the character of $\Spin_0(\Vh)$ in terms of the character of $V$.

\begin{propn} \label{p5}
{\it Let $V$ be a finite dimensional orthogonal $\g$-representation. Let $T$ be the set of all weights of $V$ and $m_{\beta}$ the multiplicity of a weight $\beta\in T$ so that the character of $V$ can be written as:
$$\begin{array} {rcl}
    \Char V&=&\sum\limits_{\beta(d_1)>0}m_{\beta}(e^{\beta}+e^{-\beta})\ +\ m_0\,.
\end{array}$$
Then $\Spin_0$ of the affinized\; $\gh$-representation $\Vh$ has the character:
$$\Char\Spin_0(\Vh)
\ =\ e^{\nu+c\Lambda_0}\!\! \prod\limits_{\beta(d_1)>0} (1\sh+e^{-\beta})^{m_{\beta}}\ \,
\mathop{\prod_{k>0}}_{\beta\in T} (1\sh+e^{-\beta-k\delta})^{m_{\beta}} \,,
$$
where $\nu=\frac{1}{2}\sum_{\beta(d_1)>0}m_\beta\beta$ and $c=\frac{1}{2}\sum_{\beta(d_1)>0}m_\beta\beta(\theta^\vee)^2\,,$ called the level of $\Spin_0(\Vh)$. Here $\theta$ is the highest root of $\g$}.
\end{propn}

\subsubsection{Factorization Theorems for affine Lie algebras}

If we restrict an affinized $\gh$-representation $\Spin_0(\Vh)$ to ${\g}{\oplus}{\C\dh}$ and apply Proposition \ref{p5}, we obtain:
\begin{propn}  \label{p6}
{\it $\Spin_0(\Vh)$, when restricted from $\gh$ to ${\g}{\oplus}{\C\dh}$\,, factors into an infinite tensor product:
$$
\Spin_0(\Vh){\downarrow}^{\gh}_{{\g}{\oplus}{\C\dh}}\ \cong\ \Spin_0(V)\;\;\otimes\;\;\wedge^\bullet (tV)\;\;\otimes\;\;\wedge^\bullet (t^2V)\;\;\otimes\cdots\,,
$$}
\end{propn}

\begin{remark}
The ${\g}\oplus{\C\dh}$-representation  $U_k:=\wedge^{\bullet}(t^kV)$
contains a canonical one-dimensional representation $\C 1=\wedge^0
t^kV$. The infinite tensor product above is the direct limit of the
maps:
$$\begin{array}{ccl}
U_0\otimes U_1\otimes\cdots\otimes U_k&\rightarrow& U_0\otimes U_1\otimes\cdots\otimes U_k\otimes U_{(k+1)}\\ u_0\otimes u_1\otimes\cdots\otimes u_k&\mapsto    & u_0\;\otimes\; u_1\;\otimes\;\cdots\;\otimes\; u_k\otimes 1
\end{array}$$
where $U_0:=\Spin_0(V)$.
\end{remark}
We now work out Theorems \ref{t1} and \ref{t2} for Affine Lie algebras.
\begin{propn}  \label{p7}
{\it The $\gh$-representation $V(\rhoh)$, when restricted to ${\g}{\oplus}{\C\dh}$\,, factors into an infinite tensor product:
$$
V(\rhoh){\downarrow}^{\gh}_{{\g}{\oplus}{\C\dh}}\cong V(\rho)\;\;\otimes\;\;\wedge^\bullet (t\g)\;\;\otimes\wedge^\bullet (t^2\g)\;\;\otimes\cdots\,.
$$
}
\end{propn}
\begin{propn} \label{p8}
{\it If we let:
 $$V(\muh){\downarrow}^\gh_{{\g}{\oplus}{\C\dh}} \,\cong\,\bigoplus_{i}V(\mu_{i})\,,$$
then:
$$V(2\muh{+}\rhoh){\downarrow}^\gt_{{\g}{\oplus}{\C\dh}}\,\cong\,\left(\bigoplus_{i}V(2\mu_{i}{+}\rho)\right)\otimes\;\wedge^{\bullet}t\g\otimes \wedge^{\bullet}t^2\g\otimes\cdots\;.
$$
}
\end{propn}
\subsection{Classification of coprimary representations}\label{CoP}

Motivated by Proposition \ref{p6}, we ask: For which representations $V$ is $\Spin_0(\Vh)$ irreducible?
\begin{defn}  A $\g$-representation $V$ is {\it coprimary} if $\Spin_0(V)$ is irreducible. \end{defn}

Panyushev \cite{P} gives a complete list of coprimary representations $V$ of a simple Lie algebra $\g$ and deduces the classification for a semi-simple Lie algebra.

\begin{propn}\label{p9}
{\it \;Let $V$ be an orthogonal representation of a finite dimensional simple Lie algebra $\g$. Then $V$ is coprimary i.e. $\Spin_0(V)$ is irreducible if and only if  $V$ is itself irreducible and is one of the following :
\begin{enumerate}
\item $V(\theta)$,\;\;\;\;\;\;for all $\g$\;\;\;\;\;\;\;\;\;\;\;\;\;\;\;\;\;\;\;\;\;\;\;\;\;\;\;\;\;\;\;\;\;\;\;\;where\;\;\;$\Spin_0(V)=V(\rho)$;
\item $V(\theta_s)$,\;\;\; for        $\g\in\left\{\mathfrak{so}_{2n+1}\C,\;\mathfrak{sp}_{2n}\mathbb{C},\;\mathfrak{f}_4\right\}$\;\;\;\;where\;\;\;$\Spin_0(V)=V(\rho_s)$;
\item $V(2\theta_s)$,\; for $\g=\mathfrak{so}_{2n+1}\mathbb{C}\;\;\;(n\geq1)$\;\;\;\;\;\;\;\;\;\;where\;\;\;$\Spin_0(V)=V(2\rho_s+\rho)$;
\end{enumerate}
where $\theta_s\;=$ highest short root of\; $\g$.}
\end{propn}

\begin{remark}
For $n=1$, we have $\mathfrak{so}_{2n+1}\C\cong\mathfrak{sl}_{2}\C$\,, and we take $\theta_s:=\theta$.
\end{remark}
The classification of coprimary affinized representations is as follows:

\begin{propn}  \label{p10}
{\it For a representation $\Vh\in\aff$ of\; $\gh$ obtained from a representation $V$ of a simple Lie algebra $\g$,
$$\Vh \text{ is coprimary } \Longleftrightarrow
\left(\begin{array}{c}
 V \text{ is coprimary and belongs to} \\
\text{ cases 1 or 2 of Proposition \ref{p9}}
\end{array}\right)$$}
\end{propn}
Panyushev proves the irreducibility of $\Spin_0(V)$ using the Weyl
denominator identity for the Langlands dual of $\g$, and analogously
we can prove the irreducibility of $\Spin_0(\Vh)$ in Proposition
\ref{p10} using the Weyl denominator identity for the Langlands dual
of $\gh$, a (possibly twisted) affine Lie algebra.

\section{General Spin construction for augmented symmetrizable Kac-Moody algebras }\label{Spin}
Next we give the construction of $\Spin_0$ for representations of augmented symmetrizable Kac-Moody algebras. This surprisingly delicate matter has been briefly studied by Kac and Peterson \cite{KP} and Pressley and Segal \cite[Chapter 12]{PS}. We provide here a different and a more detailed presentation. Also, we will do this in a more general setting which is compatible with restriction of representations.

Let $V$ be a vector space with basis $\left\{e_i:i\in I\right\}$
where the index set can be finite,
$I=\left\{m,\ldots,1,0,-1,\ldots,-m\right\}$ or
$\left\{m,\ldots,1,-1,\ldots,-m\right\}\,;$ or infinite, $I=\Z$ or
$\Z\setminus\left\{0\right\}$\,. Define a symmetric bilinear form
on $V$ by $Q(e_i\,,e_j):=\delta_{i,-j}$\;.
\subsection{Finite dimensional case}\label{FinCase}
First, let $V$ be finite dimensional ($I$ is finite). The {\it orthogonal Lie algebra $\sov$} is defined to consist of matrices which are skew-symmetric with respect to the anti-diagonal i.e.
$$
\sov:=\left\{A=(a_{i,j})_{i,j\in I}: a_{i,j}=-a_{-j,-i}\right\}.
$$
Thus, $\sov$ has a basis $\left\{Z_{i,j}:=E_{i,j}-E_{-j,-i}\;: i,j\in I,\;i>-j\right\}$ where $E_{i,j}$ are the coordinate matrices. We define the {\it Clifford algebra} $C(V,Q)$ as the associative algebra with $1$ generated by all $v\in V$ with defining relations $e_{-i}e_{j}=-e_{j}e_{-i}\;+\;2\delta_{ij}\;\;\forall\;\;i,j\in I.$ There is an embedding of Lie algebras defined by\;:
$$\begin{array}{rcl}
\phi_{F}:\sov&\longrightarrow &\;C(V,Q)\\
Z_{i,j}\;\;&\longmapsto&\frac{1}{4}(e_{-i}e_{j}-e_{j}e_{-i}).
\end{array}$$
Now, the Clifford  algebra has an action on a {\it wedge space} $\Spin(V):=\wedge^\bullet V^+$, which on generators $\left\{e_i:i\in I\right\}$ of $C(V,Q)$  is as follows : Define $I^+=\left\{i\in I:i>0\right\}$. Then $\wedge^\bullet V^+$ has a basis $\left\{e_J:=e_{j_1}\wedge e_{j_2}\wedge\ldots\wedge e_{j_k}\right\}$ for $0\leq k\leq \left|\;I^+\right|$ and $J:=\left\{j_1>j_2>\ldots>j_k\right\}\subset I^+$. Here $\left|A\right|:=\#(A)$. For $i\in I^+$ define:
$$e_i(e_J):=e_i\wedge e_J,\;\;J\neq \left\{\right\};\;\;\;\;\;\;e_{-i}(e_J):=\left\{\begin{array}{lc}
                      \epsilon(i,J)\;e_{(J\setminus\left\{i\right\})} & \text{if}\;\; i\in J\\
                                          0                           & \text{if}\;\; i\notin J
                          \end{array}\right. ,$$
where $\epsilon(i,J):=2(-1)^{|\;\left\{j\in J\;:\; j>i\right\}|}$. Also, $1(e_J):=e_J$, $e_i(e_{\left\{\right\}}=1)=e_i$ and $e_0(e_J):=(-1)^{|J|}e_J$ if $0\in I$. Finally, due to the embedding $\phi_F$ defined earlier, this action of $C(V,Q)$ induces an action of $\sov$ on $\wedge^\bullet V^+$  which is called {\it Spin representation} of orthogonal Lie algebra $\sov$.

As described in Section\;\ref{Props}, for an orthogonal $\g$-representation $V$, $\g\subset\sov$ and $\g$ acts on $\Spin(V)$ by restriction. It is easy to find the character of $\Spin(V)$(see \cite{P}) as a $\g$ representation. It turns out that if $m_0$ is the dimension of the zero weight space of $V$ then $\Spin(V)$ is isomorphic to the direct sum of $2^{\left\lfloor m_0/2\right\rfloor}$ copies of another $\g$-representation which is defined as $\Spin_0(V)$. The character of $\Spin_0(V)$ is given in Proposition \ref{p2}.

\subsection{General case}\label{GenCase}

Now, let $V$ be infinite-dimensional ($I$ is $\Z$ or $\Z-\left\{0\right\}$) with only finite linear combinations of $\left\{e_i:i\in I\right\}$ allowed. First, we naively extend the above definitions with the following modifications: The orthogonal Lie algebra, now denoted by $\so_\infty(V)$, consist of skew-symmetric matrices with respect to anti-diagonal (as before) which have only finite number of non-zero entries in each column (skew-symmetry implies the same on the rows too), so that $\so_\infty(V)$ is closed under commutator. Clifford algebra is allowed to have {\it infinite} sums of finite products of $\left\{e_i:i\in I\right\}$. The map $\phi_F$ defined in the finite dimensional case is still an embedding of Lie algebras. The infinite wedge $\Spin(V):=\wedge^\bullet V^+$ is now an infinite dimensional vector space consisting of {\it finite} linear combinations of finite wedges of $\left\{e_i:i\in I^+\right\}$.

The action of $\left\{e_i:i\in I\right\}$ on $\wedge^\bullet V^+$, as defined in the finite case, does not extend to Clifford algebra nor induce an action of $\so_\infty(V)$. For example, for $Y:=\sum_{i\in\Z_{>0}}Z_{-i,i+1}\in\so_\infty(V)$, $\phi_F(Y)=\sum_{i\in\Z_{>0}}\frac{e_ie_{i+1}}{2}$ does not act on  $1\in\wedge^\bullet V^+$ as it leads to an infinite sum. Also, for $H:=\sum_{i\in I^+}Z_{i,i}$, an infinite diagonal matrix in $\so_\infty(V)$,\;\;$\phi_F(H)=\sum_{i\in\Z_{>0}}\frac{1-e_ie_{-i}}{2}\in C(V,Q)$ does not act on $1\in\wedge^\bullet V^+$ as $\phi_F(Z_{i,i})(1)=1/2$.

In order to resolve these two issues, next we suitably modify $\so_\infty(V)$ and $\phi_F$ and define a smaller Lie algebra $\sov$ and a map $\phi$ so that the image of $\sov$ under $\phi$ does act on $\wedge^\bullet V^+$.

The Lie algebra $\sov$ consists of matrices $A=(a_{i,j})_{i,j\in I}$ such that:
\begin{enumerate}
\item $A$ is skew-symmetric with respect to the anti-diagonal: 
$a_{i,j}=-a_{-j,-i}$\,.
\item Each column $(a_{ij})_{i\in I}$ has finitely many non-zero entries.
\item The blocks $(a_{i,-j})_{i,j>0}$ and
$(a_{-i,j})_{i,j>0}$ have finitely many non-zero entries.
\end{enumerate}
Define the map:
$$\begin{array}{rcl}
\phi:\sov&\longrightarrow &\;C(V,Q)\\
Z_{i,j}\;\;&\longmapsto&-\frac{1}{2}e_{j}e_{-i}.
\end{array}$$
Now referring to the matrices $Y$ and $H$ defined earlier, note that
the matrix $Y\in \so_\infty(V)$ does not belong to $\sov$ and even
though $H$ belongs to $\sov$, $\phi(H)$ does act on
$1\in\wedge^\bullet V^+$. Further, we can verify that
image$(\phi)\subset C(V,Q)$ {\it does} act on $\wedge^\bullet V^+$.
In exchange of getting the action it turns out that the map $\phi$ is {\it not} a Lie algebra map and
image$(\phi)$ is not closed under brackets in $C(V,Q)$. But the
central extension of the image$(\phi)$:
$$
\sotv:=\left\{\phi(A):A\in\sov\right\}\;\oplus\;\C 1\;\;\subset C(V,Q)\;,
$$
is a Lie algebra which also acts on $\wedge^\bullet V^+$ (as $1$
acts by identity). The Lie algebra $\sotv$ is a central extension of
$\sov$ by one dimensional center $\C1$ due to the following
exact sequence of Lie algebra maps :
$$\begin{array}{cccccccc}
0 \longrightarrow&\mathbb{C}&\longrightarrow&\sotv       &\stackrel{\pi}{\longrightarrow}&\sov&\longrightarrow& 0\\
                   & 1        &\longmapsto    & 1        &\longmapsto                    & 0&               &  \\
                 &          &               & e_je_{-i}&\longmapsto                 &-2Z_{i,j}.&&
\end{array}$$
This can be verified using the following commutator relations in
$\sotv$ and $\sov$.
$$\begin{array}{lcccccccc}
\left[e_je_{-i},e_se_{-r}\right]&=&2\delta_{i,s}e_je_{-r}&-&2\delta_{i,-r}e_je_s&+&2\delta_{j,-s}e_{-r}e_{-i}&-&2\delta_{j,r}e_se_{-i}
\end{array}$$
$$\begin{array}{lcccccccc}
\left[2Z_{i,j},2Z_{r,s}\right]&=&-4\delta_{i,s}Z_{r,j}&+&4\delta_{i,-r}Z_{-s,j}&-&4\delta_{j,-s}Z_{i,-r}&
+&4\delta_{j,r}Z_{i,s}\;.
\end{array}$$
We can prove that this extension does not split when $V$ is infinite
dimensional. Thus $\sotv$-representation $\wedge^\bullet V^+$ which
we call {\it Spin representation} does not induce an action of
$\sov$. Therefore, the orthogonal Lie algebra $\sov$, when $V$ is
infinite dimensional, does {\it not}\; have a $\Spin$
representation, but its central extension $\sotv$ does.

The above construction can also be carried out when $V$ is finite dimensional. There the extension splits as (image$(\phi_F)\oplus\C 1$) because $\phi_F$ is an embedding and $\pi\circ\phi_F=\text{I}_{\sov}$. Thus, $\sotv$-representation $\wedge^\bullet V^+$ does induce an action of $\sov$ and the resulting representation coincides with the Spin representation of $\sov$ defined earlier. Thus, the above construction is the general construction of $\Spin(V)$ for a finite or infinite dimensional vector space $V$.

Now let $\g$ be an augmented symmetrizable Kac-Moody algebra with distinguished element $d$ and $V$ a $d$-finite orthogonal $\g$-representation (as in \S\ref{FT}). Orthogonality and $d$-finiteness of $V$ lead to a map $\gt\to\sov$. Once we have the map \;$\g\to\sov$, using $\pi:\sotv\to\sov$ defined earlier, we get an induced map $\g\to\sotv$ for any augmented symmetrizable Kac-Moody algebra due to the following lemma.

\begin{lemma}\label{l1}
{\it  Let  $\g$ be an augmented symmetrizable Kac-Moody algebra with Cartan subalgebra $\h$. Fix a complementary subspace $\h^{''}$ to $\h^{'}:=\bigoplus_{i=1}^n\C\alpha_i^\vee$ in $\h$. Thus :
$$
\g=\h^{'}\oplus\h^{''}\oplus\g_R,
$$
where $\g_R$ is the space spanned by all roots. Fix
$\psi\in(\h^{''}\oplus\g_R)^*$ such that $\psi(\g_R)=0$. Then for
any Lie algebra map $\sigma:\g{\longrightarrow}\sov$ there exists a
unique lifting $\sigt:\g{\longrightarrow}\sotv$ such that the
following diagram commutes:
$$\begin{array}{cc}
    \g\stackrel{\sigma}{\longrightarrow}&\sov\\
        \;\;\sigt{\searrow}                    &\uparrow\pi\\
                                         &\sotv
\end{array}$$
and $\sigt=\phi\circ\sigma\;+\;\psi$\; on\; $\h^{''}\oplus\g_R$. Recall that $\sotv$ was defined as $\phi(\sov)\oplus\C 1$.}
\end{lemma}

Now since $\sotv$ acts on $\Spin(V)$, that is:
$$\begin{array}{cl}
    \g\stackrel{\sigma}{\longrightarrow}&\sov\\
        \sigt\searrow                    &\uparrow\pi\\
                                         &\sotv\longrightarrow\text{End}_\C\Spin(V)\,,
\end{array}$$
for a given orthogonal $d$-finite $\g$-representation $V$, we can
define the $\g$-representation $\Spin(V)$ . In \S \ref{Props}, we
defined $\Spin_0(V)$ and some basic properties of $\Spin(V)$ and
$\Spin_0(V)$ are listed.
\section{Proofs}\label{Proofs}
\subsection{Proof of Lemma \ref{l1}}
Let $X_{\pm i},i=1\cdots n$ be the simple root vectors of $\g$ and $\left\{d_1,d_2,\cdots,d_l\right\}$ be a basis of $\h^{''}$. Then, $X_{\pm i}$'s and $d_j$'s generate $\g$ as a Lie algebra.

Note that by the commutativity of the diagram and due to the map $\psi$, the map $\sigt$ is uniquely defined on the generators of $\g$. Now, to be able to extend this map $\sigt$ to whole of $\g$, we need that $\sigt(X_{\pm i})$ and $\sigt(d_j)$ in $\sotv$ satisfy the defining bracket relations of $\g$. But since $\sigma(X_{\pm i})$ and $\sigma(d_j)$ in $\sov$ satisfy the defining bracket relations of $\g$ (as $\sigma$ is a Lie algebra map), and $\pi:\sotv\longrightarrow\sov$ is a Lie algebra map mapping $\sigt(X_{\pm i}),\;\sigt(d_j)$ to $\sigma(X_{\pm i}),\;\sigma(d_j)$ in $\sov$, we can prove that $\sigt(X_{\pm i}),\;\sigt(d_j)$ also satisfy each defining bracket relation of $\g$ up to a constant because $ker(\pi)=\C$. We show that this constant is zero for each relation.

Let $\hh$ and $\hti$ be the cartan subalgebras of $\sov$ and $\sotv$ respectively so that $\hti=\pi^{-1}(\hh)$. Then the constant term in $\sigt(X_{\pm i})$, when expressed in standard basis of $\sotv$, is zero and $\sigt(d_j)\in\hti$ as $\sigma(d_j)\in\hh$. Set $\sigt(\alpha_i^\vee)=[\sigt(X_i),\sigt(X_{-i})]$ for $i=1,\cdots,n$. Then clearly, $\sigt(\alpha_i^\vee)\in\hti$ for all $i$. This defines $\sigt$ from $\h$ into $\hti$ (may not be injective). For any $H\in\h$ we can easily conclude that $[\sigt(H),\sigt(X_{\pm i}]-\alpha_i(H)\sigt(X_{\pm i})$ is a constant. This constant must be zero because the constant term in $\sigt(X_{\pm i})$ is zero and $\sigt(H)\in\hti$. Now for $i\neq j$, $X:=[\sigt(X_i),\sigt(X_{-j})]$ must be constant and for all $H\in\h$, $\sigt(H)$ acts diagonally on $X$ with eigenvalue $(\alpha_i-\alpha_j)(H)$. This implies that $X=0$. Similarly, the generators satisfy the last bracket relation also (see \S\ref{Bkground}).\qedhere

\subsection{Proof of Propositions \ref{p1}-\ref{p2}}
First, we compute the character of $\Spin(V)$ as a
$\sotv$-representation. As in \S\ref{Spin}, let $V$ be the vector
space with basis $\left\{e_i:i\in I\right\}$. We can check that
$\hti:=\bigoplus_{i\in\;I^+}\C e_ie_{-i}\bigoplus\C 1$ is a Cartan
subalgebra of $\sotv$. Consider the dual basis of
$\left\{1,\frac{e_je_{-j}}{2}:j\in\;I^+\right\}$ i.e.
$\Lt_i\in\hti^{\ast}$ \;defined as :
$$\begin{array}{clr}
                                      \Lt_0(1)=1 & \Lt_0(\frac{e_je_{-j}}{2})=0            &  j\in I^+\;\\
                                      \Lt_i(1)=0 & \Lt_i(\frac{e_je_{-j}}{2})=\delta_{i,j} &  i,j\in I^+.
                                      \end{array}$$
A basis of weight vectors of $\wedge^{\bullet}V^+$ is
$\left\{e_J:J\subset\;I^+\right\}$, where $e_J:=e_{j_1}\wedge
e_{j_2}\wedge\cdots\wedge e_{j_k}$, for
$J=\left\{j_1,j_2,\cdots,j_k\right\}$. Observe that :
$$\begin{array}{rcl}
                   1(e_J)&=&e_J\\
 \frac{e_je_{-j}}{2}(e_J)&=&\left\{\begin{array}{ccc}
                                            e_J &\text{if}& j\in J \\
                                            0   &\text{if}& j\notin J.
                                           \end{array}\right. \;
 \end{array}$$
So, $e_J$ has the weight $\Lt_0+\sum\limits_{j\in J}\Lt_j$.
Therefore,
$$\begin{array}{ccc}
     \Char\Spin(V)=\Char(\wedge^{\bullet} V^+)&=&\sum\limits_{J\subset\;I^+}e^{(\Lt_0+\sum\limits_{j\in J}\Lt_j)},\\
                            &=&e^{\Lt_0}\prod\limits_{i\in\;I^+}(1+e^{\Lt_i}).
\end{array}$$
as $\sotv$-representation. Due to the map, $\g\stackrel{\sigt}{\longrightarrow}\sotv$, (described in \S\ref{Spin} Lemma \ref{l1}), $\Spin(V)$ becomes a $\g$-representation. Then, $\Lt_i\circ\sigt,\;i\in (I^+\cup\left\{0\right\})$, are the weights of $\Spin(V)$ as a $\g$-representation. Due to the commutativity of the diagram in lemma \ref{l1} we can show that $\Lt_i\circ\sigt=-L_i\circ\sigma,\;i\in\;I^+$, where $L_i$'s are the weights of the defining representation $V$ of $\sov$. Let $\beta_i:=L_i\circ\sigma,\;i\in\;I^+$ and $\Lambda:=\Lt_0\circ\sigt$. Thus, as a $\g$-representation,
$$\Char\Spin(V)=e^{\Lambda}\prod\limits_{i\in\;I^+}(1+e^{-\beta_i})\;.$$
Without loss of generality, we may assume $\beta_i(d)\geq 0$ for all $i\in I^+$. When $V$ is finite dimensional (see \cite{P}), $\Lambda=\frac{1}{2}\sum_{i\in I^+}\beta_i.$

\subsubsection{Proof of Proposition \ref{p1}.}
Now we prove parts (1)\;-\;(7) of Proposition \ref{p1}.
\begin{proof}[Proof of Proposition \ref{p1}(1)]
We show the following:\\
$V$ is $d$-finite\\
$\Rightarrow$ $\Spin(V)$ is $d$-finite and the set $\left\{\gamma(d):\gamma \text{\;is a weight of}\; \Spin(V)\right\}$ is bounded above.\\
$\Rightarrow$ $\Spin(V)\in\Ow$.

Let $V$ be $d$-finite. Result is obvious if $V$ is finite dimensional. So, let $V$ be a infinite dimensional so that $I=\Z$, $I^+=\Z_{>0}$ and the set of positive weights are $\left\{\beta_i:i\in\Z_{>0}\right\}$. Any weight of $\Spin(V)$ is of the form :\;$\Lambda-a$ where $a=\sum_{i\in\Z_{>0}}a_i\beta_i$ for $a_i=0\;\text{or}\;1$ for $i\in\Z_{>0}$ where $a_i=0$ for all but finitely many $i$'s. Let's call such a sequence, $(a_i)_{i\in\Z_{>0}}$, an $(a)$-sequence.

For $d$-finiteness of $\Spin(V)$, it's enough to show that the above character when restricted to $\C d$ has finite coefficients. This is equivalent to: For each $N$, there are finitely many $(a)$-sequences such that $(\Lambda - a)(d)=N$. Define
$$
L_k:=\left\{i\in\Z_{>0}:\beta_i(d)=k\right\},
$$
$L_k$ is a finite set due to $d$-finiteness of $V$.
\begin{eqnarray}
(\Lambda - a)(d)&=&N\nonumber\\
\Rightarrow\hspace{5pt} \Lambda(d)-N&=&\sum_{i\in\Z_{>0}}a_i\beta_i(d)\nonumber\\
\Rightarrow\hspace{5pt} \Lambda(d)-N&=&\sum\limits_{k\in\Z_{\geq 0}}\left(\sum\limits_{\beta_i(d)=k}a_i\right)k. \nonumber\\
\Rightarrow\hspace{5pt} \Lambda(d)-N&=&\sum\limits_{k\in\Z_{\geq 0}}\left(\sum\limits_{i\in L_k}a_i\right)k.\nonumber\\
\Rightarrow\hspace{5pt} \Lambda(d)-N&=&\sum\limits_{k\in\Z_{\geq 0}}b_k k,\label{part}
\end{eqnarray}
where $b_k:=\sum_{i\in L_k}a_i$. In the above equation $k\in\Z_{\geq 0}$ because $\beta_i(d)\geq 0$ for all $i$ and the sum is a finite sum as $a_i\neq 0$ only for finitely many $i$'s. Clearly, $b_k$ is finite. Thus each $(a)$-sequence $(a_i)_{i\in\Z_{>0}}$ satisfying equation\;(\ref{part}) gives a non-negative integral partition of $\Lambda(d)-N$ where each non-negative integer is repeated $b_k$ times. Conversely, for every such partition given by $(b_k)_{k\in\Z_{\geq 0}}$ with all but finite number of $b_k$ to be zero, there exists only finitey many $(a)$-sequences, $(a_i)_{i\in\Z_{>0}}$ such that $b_k=\sum_{i\in L_k}a_i$. Since there are finitely many such partitions of $\Lambda(d)-N$, there are finitely many $(a)$-sequences satisfying equation (\ref{part}). Thus $\Spin(V)$ is $d$-finite.

Finally, $\left\{\gamma(d):\gamma \text{\;is a weight of}\; \Spin(V)\right\}$ is bounded above by $\Lambda(d)$ as $\gamma=\Lambda-a$ for $a=\sum_{i\in\Z_{>0}}a_i\beta_i$ and $\beta_i(d)\geq 0$. That proves the first implication.

To prove $\Spin(V)\in\Ow$, we show that if $W$ is a $d$-finite representation, such that $\left\{\gamma(d):\gamma \text{\;is a weight of}\; W\right\}$ is bounded above then $W\in\Ow$. First of all weight spaces of $W$ are finite dimensional by $d$-finiteness of $W$. So, let $\Pc(W)$ denote the set of all weights of $W$. For $\beta\in\Pc(W)$, let $\gamma\in S_\beta:=\left\{\gamma\in\Pc(W):\beta\leq\gamma\right\}$. That is, $\gamma-\beta=\sum c_\alpha\alpha$, for some $c_\alpha\in\Z_{\geq 0}$ where the sum is over simple positive roots of $\g$. This implies $\gamma(d)-\beta(d)\geq 0$ as $\alpha(d)>0$ by definition of $d$. As $\gamma(d)$ is bounded above, $K\leq\gamma(d)\leq N$ where $K=\beta(d)$. So $S_\beta$ is finite, otherwise $\bigoplus_{K\leq\gamma(d)\leq N}W^{(\gamma)}$ will be infinite dimensional contradicting the $d$-finiteness of $W$ (see \S\ref{FT}). Now the set of all maximal weights of the non-empty finite set $S_\beta$ is non-empty and intersects $M_W$ non-trivially. Thus $W\in\Ow$.
\qedhere
\end{proof}
\begin{proof}[Proof of Proposition \ref{p1}(2)]

The result is obvious if $V$ is finite dimensional. Assume that $V$ is infinite dimensional. Let $X$ be a positive or negative root vector of $\g$. We will show that $X$ acts locally nilpotently on $\Spin(V)$ if it acts locally nilpotently on $\sov$ (see \S\ref{Spin}).

First assume $X$ to be a positive root vector. We denote the matrix of the action of $X$ on the representation $V$ given by the map $\g\longrightarrow\sov$ by $X$ only. Recall, that $\sov$ is defined with respect to a polarization of $V=V^+\oplus V^-$. Fix an ordered basis $\left\{\cdots>e_2>e_1>e_{-1}>e_{-2}>\cdots\right\}$ of $V$, where $V^{\pm}=\bigoplus_{i\in\Z^{\pm}}\C e_i$. We may assume that $X$ is a strictly upper triangular matrix in $\sov$ with respect to above basis.  We can write $X=Z+F$ such that $Z(V^{\pm})\subset V^{\pm}$, $F(V^{\pm})\subset V^{\mp}$.  Then by definition of $\sov$, $F(V)$ is finite dimensional. Since $X$ is a upper triangular, $F(V^+)=\left\{0\right\}$. Thus, let $F:=\sum_{r,t}b_{r,t}Z_{r,-t}$, a finite sum, where $\left\{Z_{i,j}\right\}$ forms a basis of $\sov$ defined in \S\ref{Spin}. Referring to Lemma \ref{l1} in \S\ref{GenCase}, the image of $X$, say $\Xt$, in $\sotv$ can be written as: $\Xt=\Zt+\Ft$, where $\Zt=\phi(Z)$ and $\Ft=\phi(F)$ in $\sotv$ (see \S\ref{GenCase}). Here $\Ft=-1/2\sum_{r,t}b_{r,t}e_{-t}e_{-r}$. Now action of $\Zt$ and $\Ft$ on $\Spin(V):=\wedge^\bullet V^+$ is defined as:
$$
\Zt(e_{i_1}\wedge e_{i_2}\wedge\cdots\wedge e_{i_k}):=\Zh e_{i_1}\wedge e_{i_2}\wedge\cdots\wedge e_{i_k}\;\;\;+\;\;\;e_{i_1}\wedge \Zh e_{i_2}\wedge\cdots\wedge e_{i_k}\;\;\;+\;\;\;\cdots
$$
where $\Zh$ denotes the transpose of $Z$. Thus,
$$
\Zt^p(e_{i_1}\wedge e_{i_2}\wedge\cdots\wedge e_{i_k})=\sum\limits_{p_1+\cdots+p_k=p}\frac{p!}{p_1!\cdots p_k!}\Zh^{p_1}e_{i_1}\wedge \Zh^{p_2}e_{i_2}\wedge\cdots\wedge\Zh^{p_k} e_{i_k}\;.
$$
and
$$\Ft(e_{i_1}\wedge\cdots\wedge e_r\wedge\cdots\wedge e_t\wedge\cdots\wedge e_{i_k}):=\sum\limits_{r,t}c_{r,t}e_{i_1}\wedge\cdots\wedge \hat{e}_r\wedge\cdots\wedge \hat{e}_t\wedge\cdots\wedge e_{i_k}\;.
$$
where $\hat{e}_i$ denotes the absence of $e_i$ and $c_{r,t}=\pm 2b_{r,t}$ depending on positions of $e_r$ and $e_t$\;. Also, $c_{r,t}=0$ if $e_r$ and/or $e_t$ do not occur in the wedge. Let
$$
V_m:=\C e_1\oplus\C e_2\oplus\cdots\oplus\C e_m \;.
$$
for $m\in\Z_{>0}$ and let
$$
\Zh^\bullet V_m:=\bigoplus\limits_{k\in\Z_{\geq 0}}\Zh^k(V_m)\;.
$$
By definition of $Z$, above is finite direct sum as $\Zh^k(V_m)=0$ for large $k$, say $q$.

\begin{lemma} \label{l2}
{\it For any given $K\in\Z_{>0}$ and a matrix $Z$ in $\sov$ satisfying $Z(V^{\pm})\subset V^{\pm}$ and $Z^q(V_m)=0$ for some $q$, $\Zt^P[\wedge^k(\Zh^\bullet V_m)]=0$ for all $k\leq K$ and $P=K(q-1)+1$.}
\end{lemma}
\begin{proof}
Let $a:=e_{i_1}\wedge e_{i_2}\wedge\cdots\wedge e_{i_k}\;\in\wedge^{k}(\Zh^\bullet V_m)$.
$$
\Zt^P(a)=\sum\limits_{q_1+\cdots+q_k=P}\frac{P!}{q_1!\cdots q_k!}\Zh^{q_1}e_{i_1}\wedge \Zh^{q_2}e_{i_2}\wedge\cdots\wedge\Zh^{q_k} e_{i_k}.
$$
Since, for $1\leq s\leq k$, $e_{i_s}\in\Zh^\bullet V_m$ $\Rightarrow$ $\Zh^{q_s}(e_{i_s})\in\Zh^\bullet(\Zh^{q_s} V_m)$. Now, if $q_s\leq q-1$ for all $s$ then $P=\sum_{s=1}^k q_s\leq k(q-1)\leq K(q-1)=P-1$. So, $P\leq P-1$ $\Rightarrow\Leftarrow$. Thus, $q_s\geq q$ for some $s$. And, $\Zh^{q_s}(e_{i_s})\in\Zh^\bullet(\Zh^{q_s} V_m)=0$ for some $s$ meaning $\Zt^P(a)=0$ for all $a\in\wedge^k(\Zh^\bullet V_m)$ since $a$ is arbitrary. Hence, $\Zt^P[\wedge^k(\Zh^\bullet V_m)]=0$ which is Lemma \ref{l2}.\qedhere
\end{proof}
We will prove that for a given $s\in\Z_{>0}$, we can find $N$, depending on $s$ such that:
\begin{eqnarray}
    (\Zt+\Ft)^N(\wedge^s V_m)&=&0,\label{Toprove}
\end{eqnarray}
which will prove that $\Xt=\Zt+\Ft$ acts nilpotently on $\Spin(V)=\wedge^\bullet V^+$.

Set $l:=\left\lfloor \frac{s}{2}\right\rfloor$, $P:=s(q-1)+1$ and $N:=(l+1)(P-1)+l+1$. Consider:
$$
(\Zt+\Ft)^N(\wedge^s V_m)=\bigoplus_{k=0,N}\left(\bigoplus_{\sum P_i=N-k}\Zt^{P_1}\Ft\Zt^{P_2}\cdots\Ft\Zt^{P_{k+1}}\right)(\wedge^s V_m).
$$
Verify that for any $p,k\in\Z_{>0}$, $\Zt^p(\wedge^k(\Zh^\bullet V_m))\subset \wedge^k(\Zh^\bullet V_m)$ and  $\Ft(\wedge^k(\Zh^\bullet V_m))\subset \wedge^{k-2}(\Zh^\bullet V_m)$. Thus, $$U:=\Zt^{P_1}\Ft\Zt^{P_2}\cdots\Ft\Zt^{P_{k+1}}(\wedge^s V_m)\subset \wedge^{s-2k}(\Zh^\bullet V_m),$$ where right hand is defined as zero for $s<2k$. Let $k\leq \left\lfloor \frac{s}{2}\right\rfloor=l$. Since $\sum_{i}P_i=N-k$, $P_i\geq P$ for some $i$ for a similar reason as in Lemma \ref{l2}. Let $P_j\geq P$. Now, $$\Ft\Zt^{P_{j+1}}\cdots\Ft\Zt^{P_{k+1}}(\wedge^s V_m)\subset\wedge^{s-2(k-j+1)}(\Zh^\bullet V_m).$$ By Lemma \ref{l2}, $$\Zh^{P_j}(\Ft\Zt^{P_{j+1}}\cdots\Ft\Zt^{P_{k+1}})(\wedge^s V_m)=0$$
$$\Rightarrow\;\;U:=\Zt^{P_1}\Ft\cdots\Ft\Zh^{P_j}(\Ft\Zt^{P_{j+1}}\cdots\Ft\Zt^{P_{k+1}})(\wedge^s V_m)=0$$
$$\Rightarrow\;\;(\Zt+\Ft)^N(\wedge^s V_m)=0$$
which proves $\Xt=\Zt+\Ft$ is locally nilpotent on $\Spin(V)=\wedge^\bullet V^+$.

Note that we have not used the integrability of $V$ yet. Now, for $X$ a negative root vector of $\g$, whose matrix corresponds to a strictly lower triangular matrix, the proof will require the integrability of $V$. Most of the analysis is same but is significantly different at few places.

As before $X=Z+F$. This time $F(V^-)=\left\{0\right\}$, $F:=\sum_{r,t}b_{r,t}Z_{-r,t}$ and $\Ft=-1/2\sum_{r,t}b_{r,t}e_{t}e_{r}$.
$$\Ft(e_{i_1}\wedge e_2\wedge\cdots\wedge e_{i_k}):=-\frac{1}{2}\sum\limits_{r,t}b_{r,t}\;e_t\wedge e_r\wedge e_{i_1}\wedge\cdots\wedge e_{i_k}\;.
$$
It is enough to prove equation (\ref{Toprove}) for all $m\geq\text{Max}\left\{r,t:\;b_{r,t}\neq 0\right\}$. In the previous case, we got for free the condition that $\Zh^q(V_m)=0$ for some $q$, beacuse $\Zh$ was a strictly lower triangular matrix with $\Zh(V^{\pm})\subset V^{\pm}$. In this case, we use the fact that $X$ is a locally nilpotent matrix. So, there exists a $q$ such that $(Z+F)^q(V_m^-)=0$ where
$V^-_m:=\C e_{-1}\oplus\C e_{-2}\oplus\cdots\oplus\C e_{-m}$\;. Since $X$ is skew-symmetric, this is equivalent to $(\Zh+\Fh)^q(V_m)=0$ where $\Fh$ is the transpose of $F$. Since, $F(V^-)=\left\{0\right\}$, $\Fh(V^+)=\left\{0\right\}$. In particular, $\Fh(V_m)=\left\{0\right\}$ which leads to $\Zh^q(V_m)=\left\{0\right\}$ as required.

Let $\Zh^\bullet V_m\subset V_{R}$, for some $R$ depending on $m$. We modify $l:=\left\lfloor \frac{R-s}{2}\right\rfloor$ $P:=(s+2l)(q-1)+1$ to define $N:=(l+1)(P-1)+l+1$ as before. Also, $$\Ft(\wedge^k(\Zh^\bullet V_m))\subset \wedge^{k+2}(\Zh^\bullet V_m)$$ for all $k\in\Z_{>0}$ so that $$U:=\Zt^{P_1}\Ft\Zt^{P_2}\cdots\Ft\Zt^{P_{k+1}}(\wedge^s V_m)\subset \wedge^{s+2k}(\Zh^\bullet V_m),$$ and is zero if $s+2k>R$ and the proof goes through.

The fact about $\Spin_0(V)$ follows from the character formula given before the proof of Proposition \ref{p1}.\qedhere
\end{proof}

\begin{proof}[Proof of Proposition \ref{p1}(3)]
Let $\left\{e_i:i\in I_1\right\}$ be the chosen basis of weight
vectors of $V_1$, and $\left\{e^{'}_i:i\in I_2\right\}$ of $V_2$. If
at least one of $m_1$ or $m_2$ is even, $(V_1\oplus
V_2)^+=V_1^+\oplus V_2^+$. Define the map:
$$\begin{array}{rcl}
\Spin(V_1\oplus V_2)&\stackrel{p}{\longrightarrow}&\Spin(V_1)\otimes\Spin(V_2)\\
e_J\wedge e^{'}_K&\longmapsto& e_J\otimes e^{'}_K\\
1&\longmapsto& 1\otimes 1
\end{array}$$
for $J\subset I_1^+$ and $K\subset I_2^+$, where at least one of them is not empty and $e_{\left\{\right\}}:=1$ and $e^{'}_{\left\{\right\}}:=1$. Here, $p$ is an isomorphism of $\sot (V_1)\oplus\sot (V_2)$-modules as it can be easily verified that $X\circ p=p\circ X$ for $X\in\sot (V_1)$ and $X\in\sot (V_2)$.

Now let both $m_1$ and $m_2$ be odd, so that the chosen bases of $V_1$ and $V_2$ are $\left\{\cdots,e_2,e_1,e_0,e_{-1},e_{-2},\cdots\right\}$ and
 $\left\{\cdots,e{'}_2,e^{'}_1,e^{'}_0,e^{'}_{-1},e^{'}_{-2},\cdots\right\}$ respectively. Define $u:=(e_0+ie_0^{'})/\sqrt{2}$, $v:=(e_0-ie_0^{'})/\sqrt{2}$ so that $u,v$ are paired non-degenerately with respect to the bilnear form. Now, $(V_1\oplus V_2)^+=V_1^{+}\oplus\ V_2^+\oplus\C u$. Define the map,
$$\begin{array}{rcl}
\Spin(V_1)\otimes\Spin(V_2)&\stackrel{p}{\longrightarrow}&\Spin^{\text{even}}(V_1\oplus V_2)\\
e_{J_1}\otimes e^{'}_{J_2}&\longmapsto& \frac{(-1)^{t_1t_2}}{i^{t_2}}e_{J_1}\wedge e^{'}_{J_2}\wedge(1-t+t\frac{u}{\sqrt{2}})\\
1\otimes 1&\longmapsto& 1
\end{array}$$
where, $t_k:=|J_{k}|$mod\;$2$,\;$k=1,2$\; and $t:=(|J_1|+|J_2|)$mod\;$2$ and $i:=\sqrt{-1}$. Again, $p$ is an isomorphism of $\sot (V_1)\oplus\sot (V_2)$-modules as it can be easily verified that $X\circ p=p\circ X$ for $X\in\sot (V_1)$ and $X\in\sot (V_2)$. A similar isomorphism can be given for $\Spin^{\text{odd}}(V_1\oplus V_2)$.
\end{proof}

\begin{proof}[Proof of Proposition \ref{p1}(4)] $\Spin_0(V_1\oplus V_2)\;\cong\;\Spin_0(V_1)\otimes\Spin_0(V_2)$.\\
Using Proposition \ref{p1}(3), let at least one of $m_1$ or $m_2$ is even say, $m_1$. So, let $m_1=2k_1$ and $m_2=2k_2+\epsilon$, where $\epsilon=0\;\text{ or}\; 1$. Now using Proposition \ref{p1}(2) we get:
$$2^{k_1+k_2}\Spin_0(V_1\oplus V_2)\cong 2^{k_1}\Spin_0(V_1)\otimes 2^{k_2}\Spin_0(V_2),$$
which gives the desired result.

Now let $m_1=2k_1+1$ and $m_2=2k_2+1$ then as before we get:
$$2^{k_1+k_2+1-1}\Spin_0(V_1\oplus V_2)\cong 2^{k_1}\Spin_0(V_1)\otimes 2^{k_2}\Spin_0(V_2),$$
which again gives the desired result.\qedhere
\end{proof}
\begin{remark}
Proposition \ref{p1}(4) is also true for an {\it infinite} direct sum $V:=V_1\oplus V_2\oplus\cdots$ if $V$ is $d$-finite. Also, we define the infinite tensor product on the right hand side as: Let
$$
U_k:=\Spin_0(V_1)\;\otimes\;\cdots\;\otimes\;\Spin_0(V_k)\;\otimes\;\Spin_0(V_{k+1}\oplus V_{k+2}\oplus\cdots).
$$
There is an isomorphism $U_k\to U_{k+1}$ by Proposition \ref{p1}(4). The infinite tensor product, $\Spin_0(V_1)\otimes\Spin_0(V_2)\otimes\cdots,$
is defined as direct limit of maps $U_k\to U_{k+1}$.
\end{remark}
\begin{proof}[Proof of Proposition \ref{p1}(5)]
Let $W$ be root finite. So, there exists only
finitely many weights of $W$ outside the root cone $C$ defined in
\S\ref{Props} and $W$ has finite dimensional weight spaces. Hence,
for $d$-finiteness of $W$, it is enough to prove that the sets
$S_k^\pm:=\left\{\beta\in\Pc\cap
C^\pm:\beta(d)=k\right\}\supset\left\{\beta\in\Pc(W)\cap
C^\pm:\beta(d)=k\right\}$, is finite for each $k\in\Z^\pm$, where
$\Pc$ denotes the weight lattice of $\g$, $\Pc(W)$ are the weights
of $W$ and $C^\pm:=\left\{\sum_{j=1}^n
c_j\alpha_j\;:\;c_j\in\R^\pm\cup\left\{0\right\}\right\}$.

Let $\Pc_R$ be the root lattice,
$\Pc/\Pc_R=\left\{\Pc_1,\Pc_2,\cdots\right\}$. There exists finitely
many $i$'s for which $\Pc_i\cap C^+\neq\left\{\right\}$, say
$i=1,\cdots,m$. For $1\leq i\leq m$, let $p_i\in\Pc_i\cap C^+$ be a
coset representative of $\Pc_i$ such that for each
$\beta\in\Pc_i\cap C^+$ can be written as $\beta=p_i+\sum_{j=1}^n
c_j\alpha_j$ with $c_j\in\Z_{\geq 0}$. Then
$$\left\{\beta\in\Pc_i\cap C^+: \beta(d)=k\right\}=\left\{p_i+\sum_{j=1}^n c_j\alpha_j : p_i(d)+\sum_{j=1}^n c_j\alpha_j(d)=k, \;c_j\in\Z_{\geq 0}\right\},$$
which is clearly a finite set using the definition of $d$. This
leads to finiteness of $S_k^+$. Similarly we prove that $S_k^-$ is
finite.\qedhere
\end{proof}

\begin{proof}[Proof of Proposition \ref{p1}(6)]

Let $V$ be $d$-finite and orthogonal. To prove:
$$
V \;\text{is root-finite}\Leftrightarrow\Spin(V)\in\O .
$$
($\Rightarrow$)

\noindent Use the character formula given before the proof of
Proposition \ref{p1},
$$\Char\Spin(V)=e^{\Lambda}\prod\limits_{i\in\;I^+}(1+e^{-\beta_i}).$$
Refer to Proof of Proposition \ref{p1}(5) for the definitions of $C,
C^+,C^-,\Pc,\PR,\;p_j,\Pc_j, j=1\cdots m$ and define the sets:
$$I_1:=\left\{i\in I^+:\beta_i\in\Pc\cap C^+\right\},$$
$$I_2:=\left\{i\in I^+:\beta_i\in\Pc- C\right\}.$$
Note that $I_1\cup I_2=I^+$ as $\beta_i\notin C^-$ because by assumption $\beta_i(d)>0$ for all $i\in I^+$. By root finiteness, $I_2$ is a finite set.

Now let $M$ denote the set of all minimal weights (in the root order defined in \S\ref{Props}) of the finite set 
$\left\{\sum_{i\in J_2}\beta_i:J_2\subset I_2\right\}.$
Thus, $M$ is finite. Define the set of weights in $\Pc$:
$$T:=\left\{\Lambda-p_j-\gamma\;\;:\;\; j=1\cdots m,\;\; \gamma\in M\right\}, $$
which is a finite set and $p_j$'s are coset representatives defined earlier. We will show that elements of $T$ ``cover" all weights of $\Spin(V)$ in the root order.

Due to the character formula above, any weight of $\Spin(V)$ is of the form $\Lambda-\sum_{i\in J}\beta_i=\Lambda-\sum_{i\in J_1}\beta_i-\sum_{i\in J_2}\beta_i,$ for some set $J\subset I^+$ and $J_1:=J\cap I_1$ and $J_2:=J\cap I_2$. Let $\lambda$ be any such weight of $\Spin(V)$. For the sum $\sum_{i\in J_1}\beta_i$ in $\lambda$, each $\beta_i$ can be written as $p_j+b_j$ for some $j$ and $b_j\in\Pc_R\cap C^+$. Now any finite sum of coset representatives $p_j$'s can be written as $p_k+a$ for $a\in\Pc_R\cap C^+$. So $\sum_{i\in J_1}\beta_i=p_k+a+\sum_{j=1}^m b_j=p_k+b$ for $b\in\Pc_R\cap C^+$. Write $c:=\sum_{i\in J_2}\beta_i$. Therefore, $\lambda=\Lambda-(p_k+b)-c$. Choose an element $\gamma\in M$ such that $\gamma\leq c$ so that $c-\gamma\in\Pc_R\cap C^+$. Set $t:=\Lambda-p_k-\gamma\in T$. Then $t-\lambda=b+c-\gamma\in\Pc_R\cap C^+$ so that $\lambda\leq t$. This proves that $\Spin(V)\in\O$.
\\[.5em]
($\Rightarrow$)

\noindent Suppose $V$ is not root finite. Refer to the character formula given at the beginning of the Proof of the Proposition \ref{p1}. Since $V$ is not root finite, there exists an infinite set $J$ such that $\beta_i\in\Pc-C$ for all $i\in J$ (See definitions of $\Pc$ and $C$ in Proof of Proposition \ref{p1}(5)).
Extend the collection of simple positive roots, $\left\{\alpha_k\right\}_{k=1}^n$, of $\g$ to a basis of $\h^*$, say:
$$\left\{\alpha_1,\cdots,\alpha_n,\alpha_{n+1},\cdots,\alpha_P\right\}.$$
For $i\in J$, let $$\beta_i=\sum_{k=1}^nc_{i,k}\alpha_k=\beta_{i,1}+\beta_{i,2}$$
where $\beta_{i,1}$ lies in the real space spanned by $\left\{\alpha_1,\cdots,\alpha_n\right\}$ and $\beta_{i,2}$ lies in the real space spanned by $\left\{\alpha_{n+1},\cdots,\alpha_P\right\}$. Because $\beta_i\in\Pc-C$, for each $i\in J$, either $\beta_{i,2}\neq 0$ or $\beta_{i,2}=0$ and there exists a $k$ such that $1\leq k\leq n$ for which $c_{i,k}<0$.

Thus, we may find an infinite set $J_1\subset J$ and a $k$ such that if $n+1\leq k\leq P$ then $c_{i,k}$ is of same sign for all $i\in J_1$ and if $1\leq k\leq n$ then $c_{i,k}<0$ for all $i\in J_1$.

Define $\gamma_K:=\sum_{i\in K} \beta_i$ for $K$ a finite subset of $J_1$. Also, let $\lambda_K:=\Lambda-\gamma_K$. Then $\lambda_K$ is a weight of $\Spin(V)$ for each $K\subset J_1$. If $n+1\leq k\leq P$ then $\lambda_K$ when expressed in terms of above basis, will have coefficient of $\alpha_k$ unbounded (above or below) when $K$ varies over finite subsets of $J_1$. If $1\leq k\leq n$ then this coefficient will be unbounded above. Thus these weights of $\Spin(V)$ can not be bounded above in root order by finitely many weights from the weight lattice. Thus $\Spin(V)\notin\O$.\qedhere
\end{proof}
\begin{proof}[Proof of Proposition \ref{p1}(7)]  $V(\rho)\;\cong\;\Spin_0(\g)$.

\noindent For the distinguished element $d$ of $\g$, the character
formula given at the beginning of the proof of Proposition \ref{p1}
leads to:
$$\Char\Spin_0(\g)=e^\Lambda\prod_{\alpha\in R^+}(1+e^{-\alpha}),$$

\noindent where the positive roots $R^+$ correspond to the
$d$-positive weights $\left\{\alpha\in R:\alpha(d)>0\right\}$ of
adjoint representation $\g$. Also, $\Lambda=\sum_{i=1}^n
c_i\Lambda_i$, for $\Lambda_i$ the $i$th fundamental weight of $\g$.
The formula above is the character formula for $V(\rho)$ if $c_i=1$.

Restrict the adjoint representation to the 4-dimensional subalgebra
$\dd_i:=\mathfrak{s}_i\oplus\C d\subset\g$\,, where
$\mathfrak{s}_i\cong \sl_2(\C)$ corresponds to the simple root
$\alpha_i$\,and decompose the adjoint representation into  finite
dimensional $\dd_i$-orthogonal-irreducibles:
$\g\downarrow\cong\bigoplus_kV_k$. Here $\downarrow$ denotes the
restriction to $\dd_i$. Thus, by Proposition \ref{p1}(4):
$$\Spin_0(\g\downarrow)\cong\Spin_0(V_1)\otimes\Spin_0(V_2)\otimes\cdots\;.$$

\noindent Using the character formula for $\Spin_0$ for
finite-dimensional representations, \cite{P}, the distinguished
weight $\lambda_k$ of $\Spin_0(V_k)$ is the half sum of $d$-positive
weights of $V_k$. Consider a highest weight vector
$v=\sum_{\beta}a_\beta X_\beta$ of $V_k$ for $\beta$'s positive
roots of $\g$. If $\beta=\sum_\alpha b_\alpha\alpha$ for $\alpha$'s
simple roots of $\g$, then $\text{ad}(X_{-\alpha_i})^{l}(X_\beta$)
is a positive root vector or zero for all $l$ unless
$\beta=\alpha_i$.  This leads to $\lambda_k(\alpha_i^\vee)=0$ if
$v\notin\g^{(\alpha_i)}$ and $\lambda_k(\alpha_i^\vee)=1$ if
$v\in\g^{(\alpha_i)}$. Thus, $c_i=\sum_k\lambda_k(\alpha_i^\vee)=1$
for all $i$.\qedhere
\end{proof}

\subsubsection{Proof of Proposition \ref{p2}.}
Refer to the character formula given before the proof of Proposition \ref{p1} which leads to:
$$\Char\Spin_0(\g)=e^\Lambda\prod_{\beta(d)>0}(1+e^{-\beta}),$$
where we can write, $\Lambda=\sum_{i=1}^n c_i\Lambda_i$, for $\Lambda_i$ the $i$th fundamental weight of $\g$. Lets call the weights $\beta$ of $V$ such that $\beta(d)>0$ as $d$-positive weights of $V$. If $V$ is finite dimensional then by \cite{P}, $\Lambda=\sum_{\beta(d)>0}\frac{1}{2}m_\beta\beta$. When $V$ is infinite dimensional, this is an infinite sum but as in proof of Proposition \ref{p1}(7), we may restrict $V$ to $\dd_i:=\mathfrak{s}_i\oplus\C d\subset\g$ and decompose $V\hspace{-.3em}\downarrow\cong\bigoplus_kV_k$ into finite dimensional $\dd_i$-orthogonal-irreducibles $V_k$. Then the distinguished weight $\lambda_k$ of $\Spin_0(V_k)$ is the half sum of $d$-positive weights of $V_k$. But for a  weight $\beta$  of $V$ if both $\beta$ and $s_i(\beta)$ are $d$-positive, that is, $\beta(d)>0$ and $s_i\beta(d)>0$ they will not contribute to $\lambda_k(\alpha_i^\vee)$ because $\beta(\alpha_i^\vee)+s_i\beta(\alpha_i^\vee)=0$. Finally, by Proposition \ref{p1}(4), $c_i=\sum_k\lambda_k(\alpha_i^\vee)$, and the result follows.\qedhere
\subsection{Proof of Theorem \ref{t1}}
Proposition \ref{p1}(7), says:
$$\Spin_0(\gt)\cong V(\rhot)$$
We restrict the adjoint representation $\gt$ to $\g$ and apply
$\Spin_0$ with respect to $\g$. Using, Proposition \ref{p1}(4), we
get the following commutative diagram:
$$\begin{array}{ccc}
\gt&\stackrel{\Spin_0}{\longrightarrow}&V(\rhot)\\
\downarrow&                            &\downarrow\\
\g\oplus\p_1\oplus\p_2\oplus\cdots&\stackrel{\Spin_0}{\longrightarrow}&V(\rho)\otimes W_1\otimes W_2\otimes\cdots
\end{array}$$
where vertical arrows denote restriction. The diagram commutes for the following reason. Fix a $d$-finite, orthogonal $\gt$-representations $V$ (such as adjoint representation of $\gt$). Clearly, $\Spin$ commutes with restriction, $\downarrow^\gt_\g$ when acted on $V$. Express, $\Spin$ in terms of $\Spin_0$ using Proposition \ref{p1}(2). Since $\g\subset\gt$ is a $d$-embedding and $V$ is a $d$-finite representation, the non-zero $\gt$-weights of $V$ restrict to non-zero $\g$-weights. Thus the dimension of the zero weight space does not change upon restriction. Hence $\Spin_0$ also commutes with $\downarrow^\gt_\g$ when applied to $V$.

\subsection{Proof of Theorem \ref{t2}}

Let $\chi_\lambda:=\Char V(\lambda)$. For any $\chi=\sum_{\lambda\in
I}c_{\lambda}e^\lambda$, $I\subset{\mathcal P}$ the weight lattice,
define $\chi^{(2)}=\sum_{\lambda\in I}c_{\lambda}e^{2\lambda}$. Weyl
character formula for character of an irreducible
$\g$-representation with highest weight $\lambda$, says:
$$\chi_{\lambda}=\frac{A_{\lambda+\rho}}{A_{\rho}},$$
where the skew-symmetrizer, $A_\mu:=\sum_{w\in\W}\text{sign}(w)e^{w(\mu)}$ and $\W$ is Weyl group.
\begin{eqnarray}
\chi_{2\mut+\rhot}&=&\frac{A_{2(\mut+\rhot)}}{A_{\rhot}}\nonumber\\
                  &=&\frac{A_{2(\mut+\rhot)}}{A_{2\rhot}}\frac{A_{2\rhot}}{A_{\rhot}}\;.\nonumber\\
\chi_{2\mut+\rhot}&=&\frac{A_{\mut+\rhot}^{(2)}}{A_{\rhot}^{(2)}}\frac{A_{2\rhot}}{A_{\rhot}}\nonumber\\
\text{Thus},\;\;\;\chi_{2\mut+\rhot}&=&\chi_\mut^{(2)}\;\;\chi_\rhot\label{chi2mutprhot}
\end{eqnarray}
Let $\chi_\mut\downarrow=\sum_i\chi_{\mu_i}$, where $\downarrow$ denotes restriction from $\gt$ to $\g$. Also, let $W:=W_1\otimes W_2\otimes\cdots$,\;where $W_k$'s are defined in Theorem \ref{t1}.
\begin{eqnarray*}
\chi_{2\mut+\rhot}\downarrow&=&\chi_\mut^{(2)}\downarrow\;\;\chi_\rhot\downarrow\\
&=&(\sum\limits_{i=1}^k\chi_{\mu_i}^{(2)})\;(\chi_{\rho}\;\Char(W))\;\;\;\;\;\text{by Theorem \ref{t1}}.\\
&=&(\sum\limits_{i=1}^k\chi_{\mu_i}^{(2)}\chi_{\rho})\;\;\Char(W)\\
\text{Therefore},\;\chi_{2\mut+\rhot}\downarrow&=&(\sum\limits_{i=1}^k\chi_{2\mu_i+\rho})\;\;\Char(W),
\end{eqnarray*}
where the last equality is due to the same reason as for equation
(\ref{chi2mutprhot}). This proves Theorem \ref{t2}. \qed
\subsection{Proof of Propositions \ref{p3}-\ref{p4}}
\subsubsection{Proof of Proposition \ref{p3}}
We will use $S_\lambda(x_1,x_2,\cdots,x_n)$ to denote irreducible characters for $\gt=\sl_n\C$ and $\chi_k$ for irreducible character of principal $\g=\sl_2\C$ with highest weight $k\varpi$. The restriction in this case corresponds to setting $x_i$, in $S_\lambda$ as $q^{(n+1-2i)/2}$ for $q=e^\alpha$, and $\alpha$ the root of $\g$. We fix the following notation:
$$S_\lambda\downarrow:=S_\lambda(q^{(n-1)/2},q^{(n-3)/2},\cdots,q^{-(n-1)/2}),$$
and
$$S_\lambda^{(2)}:=S_\lambda(x_1^2,x_2^2,\cdots,x_n^2).$$
Then character of adjoint representation of $\gt$ when restricted to $\g$ gives:
$$S_\theta\downarrow=\chi_2+\chi_4+\cdots\chi_{2(n-1)},$$
where $\theta$ is the highest root of $\gt$. We apply $\Spin_0$ on corresponding representations, and use the character formula for $\Spin_0$ from Proposition \ref{p2}. Then using Theorem \ref{t1} we obtain:
$$S_\rho\downarrow=\chi_1\cdot u_1\cdot u_2\cdots u_{n-2}$$
where $u_k:=\prod_{j=1}^{k+1}(q^{j/2}+q^{-j/2})$ and $\rho=(n-1,n-2,\cdots,1,0)$.
Now, Theorem \ref{t2}, leads to:
$$S_{2\mu+\rho}\downarrow=(S_\mu^{(2)}\downarrow\chi_1)\;\cdot u_1\cdot u_2\cdots u_{n-2}$$
where all $n-1$ factors on right hand side are characters of $\g=\sl_2\C$.
In order to translate this in language of principal specialization, we observe the following identity:
$$S_\lambda(1,q,q^2,\cdots,q^{n-1})=q^N S_\lambda\downarrow$$
where $N=\frac{n-1}{2}\sum\limits_{i=1}^n\lambda_i$, for $\lambda=(\lambda_1,\lambda_2,\cdots,\lambda_n)$.
Using the above formula, we obtain:
$$\begin{array}{l}
S_{2\mu+\rho}(1,q,q^2,\cdots,q^{n-1})\\[.5em]
\qquad=\left(\, q^{\binom{n}{3}}(1\sh+q)\,S_\mu(1,q^2,q^4,\ldots,q^{2n-2}\,)\right)\cdot w_1(q)\cdot w_2(q)\cdots w_{n-2}(q)\,,
\end{array}$$
where $w_k(q)=(1\sh+q)(1\sh+q^2)\cdots(1\sh+q^{k+1})$\,, where all $(n-1)$ factors on the right are symmetric unimodal as they have been obtained from $\sl_2\C$-characters.
\subsubsection{Proof of Proposition \ref{p4}}
We list all graph automorphisms of Dynkin diagrams, $\Dt$, of simple Lie algebras $\gt$ and the corresponding fixed subalgebras $\g$. We obtain the factorizations using Theorem \ref{t1}. Let $\theta$ and $\thetat$ be the highest roots of $\g$ and $\gt$ repectively so that $V(\theta)$ and $V(\thetat)$ denote their adjoint representations. Let $V(\thetat)\downarrow^\gt_\g\cong V(\theta)\oplus\p_1\oplus\p_2\oplus\cdots$, so that $V(\rhot)\downarrow^\gt_\g\cong V(\rho)\otimes W_1\otimes W_2\otimes\cdots$. We specify $\p_i$'s and $W_i$'s. The fact that $\Spin_0(\p_i)\cong W_i$ will be proved in Proposition \ref{p9}.
\begin{enumerate}
\item  Graph automorphisms of order 2.
   \begin{enumerate}
   \item $\gt=A_{2n-1}=\sl_{2n}\C$.\\
         $\g= C_{n}  =\spp_{2n}\C$. \\
         $i\stackrel{\phi}{\leftrightarrow}2n-i$.\\
         $V(\thetat)\downarrow^\gt_\g\cong V(\theta)\oplus V(\theta_s)$, where $\theta_s$ is the highest short root of           $\g$.\\
         $V(\rhot)\downarrow^\gt_\g\cong V(\rho)\otimes V(\rho_s).$
   \item $\gt=D_{n+1}=\so_{2n+2}\C$.\\
         $\g= B_{n}  =\so_{2n+1}\C$. \\
         $n\stackrel{\phi}{\leftrightarrow}n+1$, interchanges two forked nodes and fixes others.\\
         $V(\thetat)\downarrow^\gt_\g\cong V(\theta)\oplus V(\theta_s)$, where $\theta_s$ is the highest short root of           $\g$.\\
         $V(\rhot)\downarrow^\gt_\g\cong V(\rho)\otimes V(\rho_s).$
   \item $\gt=E_{6}$.\\
         $\g= F_{4}$. \\
         $1\stackrel{\phi}{\leftrightarrow}5$\;\;\;$2\stackrel{\phi}{\leftrightarrow}4$ and fixes others.(See page               53\;\cite{K} for ordering of nodes)\\
         $V(\thetat)\downarrow^\gt_\g\cong V(\theta)\oplus V(\theta_s)$, where $\theta_s$ is the highest short root of           $\g$.\\
         $V(\rhot)\downarrow^\gt_\g\cong V(\rho)\otimes V(\rho_s).$
   \item $\gt=A_{2n}=\sl_{2n+1}\C$.\\
         $\g= B_{n} =\so_{2n+1}\C$. \\
         $i\stackrel{\phi}{\leftrightarrow}2n+1-i$.\\
         $V(\thetat)\downarrow^\gt_\g\cong V(\theta)\oplus V(2\theta_s)$, where $\theta_s$ is the highest short root of          $\g$.\\
         $V(\rhot)\downarrow^\gt_\g\cong V(\rho)\otimes V(\rho+2\rho_s).$
   \end{enumerate}
\item Graph automorphism of order 3.
         \begin{enumerate}
   \item $\gt=D_{4}$.\\
         $\g= G_{2}$. \\
         $\phi$ cyclically pemutes the three outer nodes and fixes the middle node. \\
         $V(\thetat)\downarrow^\gt_\g\cong V(\theta)\oplus V(\theta_s)\oplus V(\theta_s)$, where $\theta_s$ is the
         highest short root of $\g$.\\
         $V(\rhot)\downarrow^\gt_\g\cong V(\rho)\otimes (V(\rho_s)+V(0))\otimes (V(\rho_s)+V(0)).$
         \end{enumerate}
\end{enumerate}
\subsection{Proof of Propositions \ref{p5}-\ref{p8}}
\subsubsection{Proof of Proposition \ref{p5}.}
We will prove Proposition \ref{p5} using Proposition \ref{p2}.
Recall that $-N<\beta(d_1),\theta(d_1)<N$ for all weights $\beta\in
T$ of $V$. Now all weights of $\Vh$ are of the form $k\delta+\beta$
with multiplicity $m_\beta$ for $k\in\Z$ and $\beta$ a weight of
$V$. Thus, for $\dh=Nd+d_1$, $(k\delta+\beta)(\dh)=kN+\beta(d_1)>0$
if and only if $k>0$ or $k=0,\;\beta(d_1)>0$ by the definition of
$N$. Proposition \ref{p2} leads to:
$$\Char\Spin_0(\Vh)
\ =\ e^{\Lambda}\!\! \prod\limits_{\beta(d_1)>0} (1\sh+e^{-\beta})^{m_{\beta}}\ \,
\mathop{\prod_{k>0}}_{\beta\in T} (1\sh+e^{-\beta-k\delta})^{m_{\beta}} \,,
$$
where $\Lambda=\sum_{i=0}^nc_i\Lambda_i$, $c_i$ as defined in Proposition \ref{p2}. It is easy to verify that for $i=1\;\cdots\; n$, $c_i=\sum\frac{1}{2}m_\beta\beta(\alpha_i^\vee)$ summing over weights $\beta$ of $V$ (as opposed to $\Vh$) such that $\beta(d_1)>0$ and $s_i(\beta)(d_1)<0$ because $s_i(k\delta+\beta)=k\delta+s_i(\beta)$. Futher, since $V$ is finite dimensional, we may drop the condition $s_i(\beta)(d_1)<0$ and sum over all weights $\beta$ of $V$ such that $\beta(d_1)>0$ because $\beta(\alpha_i^\vee)+s_i(\beta)(\alpha_i^\vee)=0$. Thus:
$$c_i=\sum_{\beta(d_1)>0}\frac{1}{2}m_\beta\beta(\alpha_i^\vee),\;\;i=1,\cdots,n$$

For $i=0$ case, $(k\delta+\beta)(\alpha_0^\vee)=-\beta(\theta^\vee)$ as $\alpha_0=K-\theta^\vee$. Also, $m_{k\delta\pm\beta}=m_\beta$. Therefore by replacing $\beta$ by $-\beta$, we get $c_0=\sum\frac{1}{2}m_{\beta}\beta(\theta^\vee)$, summing over all $k\in\Z$ and $\beta\in T$ such that $(k\delta+\beta)(\dh)>0$ and $s_0(k\delta+\beta)(\dh)<0$ which simplifies to the inequality:
\begin{eqnarray}
    \frac{\beta(d_1)}{N}<&k&<\;\;\;\beta(\theta^\vee)+\frac{s_\theta\beta(d_1)}{N}\label{Ineq}
\end{eqnarray}
where $s_\theta$ denotes the reflection corresponding to the highest
root $\theta$ of $\g$.

Now, by definition of $N$, $\frac{\beta(d_1)}{N}$ and
$\frac{s_\theta\beta(d_1)}{N}$ are fractions. So the inequality
(\ref{Ineq}) implies that $\beta(\theta^\vee)\geq 0$. Since $c_0$
involves summing $\frac{1}{2}m_\beta\beta(\theta^\vee)$ over
inequality (\ref{Ineq}), we may sum over $\beta(\theta^\vee)> 0$.
Consider the following cases:
$$
\begin{array}{cl}
    \text{Case 1}: \beta(\theta^\vee)>0, \;\beta(d_1)<0\;&(\Rightarrow \;\;s_\theta\beta(d_1)<0).\\
    \text{Case 2}: \beta(\theta^\vee)>0, \;\beta(d_1)>0\;&\text{and}\;\; s_\theta\beta(d_1)>0.\\
    \text{Case 3}: \beta(\theta^\vee)>0, \;\beta(d_1)>0\;&\text{and}\;\; s_\theta\beta(d_1)<0.
\end{array}
$$
In Case 1, inequality (\ref{Ineq}) $\Rightarrow 0\leq k\leq\beta(\theta^\vee)-1$. In Case 2, inequality (\ref{Ineq}) $\Rightarrow 1\leq k\leq\beta(\theta^\vee)$ and in Case 3, it implies $ 1\leq k\leq\beta(\theta^\vee)-1$. Thus, we get:
\begin{eqnarray*}
c_0&=&\sum_{\text{Case 1}}\frac{1}{2}m_\beta\beta(\theta^\vee)^2+\sum_{\text{Case 2}}\frac{1}{2}m_\beta\beta(\theta^\vee)^2+\sum_{\text{Case 3}}\frac{1}{2}m_\beta(\beta(\theta^\vee)^2-\beta(\theta^\vee))\\
&=&\sum_{\text{Case
1,2,3}}\frac{1}{2}m_\beta\beta(\theta^\vee)^2-\sum_{\text{Case
3}}\frac{1}{2}m_\beta\beta(\theta^\vee)\;.
\end{eqnarray*}
Now, in the first sum the union of the three cases leads to the case $\beta(\theta^\vee)>0$ and due to the square in the sum, it is equivalent to summing over $\beta(d_1)>0$. In the second sum over Case 3, we may drop $(\beta(\theta^\vee)>0)$ as it is implied by $\beta(d_1)>0\;\text{and}\;\; s_\theta\beta(d_1)<0$. Further, we may also drop $(s_\theta\beta(d_1)<0)$ for the same reason which led to the expression for $c_i,\; i=1\cdots n$. Therefore:
\begin{eqnarray*}
c_0&=&\sum_{\beta(d_1)>0}\frac{1}{2}m_\beta\beta(\theta^\vee)^2-\sum_{\beta(d_1)>0}\frac{1}{2}m_\beta\beta(\theta^\vee)\;,\\
c_i&=&\sum_{\beta(d_1)>0}\frac{1}{2}m_\beta\beta(\alpha_i^\vee)\;\;\;\;\;i=1,2\cdots
n\;,
\end{eqnarray*}
and $\Lambda=\sum_{i=0}^nc_i\Lambda_i$
which leads to Proposition \ref{p5} using $\Lambda_i=\varpi_i+a_i^\vee\Lambda_0$ and $\theta^\vee=\sum_{i=1}^na_i^\vee\alpha_i^\vee$. Here, $\varpi_i$ is the $i$th-fundamental weight of $\g$.
\subsubsection{Proof of Proposition \ref{p6}.}
Exactly same as that of Theorem \ref{t1} using property of $\Spin_0$ given in Proposition \ref{p1}(4).
\subsubsection{Proof of Proposition \ref{p7}.}
Direct consequence of Propositions \ref{p6} and \ref{p1}(7).
\subsubsection{Proof of Proposition \ref{p8}.}
This is just Theorem \ref{t2} for affine Lie algebras where we make use of Proposition \ref{p7}.
\subsection{Proof of Propositions \ref{p9}-\ref{p10}}
We will use $R,R_s$ and $R_l$ to denote the set of all roots, short roots (if any) and long roots (if any) of a finite-dimensional semi-simple Lie algebra $\g$ with distinguished element $d=\rho^\vee$, the sum of all fundamental co-weights of $\g$. Similarly $R^+,R_s^+$ and $R_l^+$ will denote the set of all positive roots, positive short roots and positive long roots.
\subsubsection{Proof of Proposition \ref{p9}.}
This classification was done by \cite{P}. Here we give proofs of the facts about $\Spin_0(V)$ for each of the 3 cases.
\begin{proof}[Proof for Case 1]
It is given in Proposition \ref{p1}(7) earlier which uses Weyl denominator identity. We use it again to prove other cases also.
\end{proof}

\begin{proof}[Proof for Case 2]
Let $\chi:=\Char\Spin_0 V(\theta_s)$. By Proposition \ref{p2}:
\begin{eqnarray*}
    \chi&=&e^{\rho_s}\prod\limits_{\alpha\in R_s^+}(1+e^{-\alpha}).
\end{eqnarray*}
Weyl denominator identity is:
\begin{eqnarray*}
A_\rho&=&e^{\rho}\prod\limits_{\alpha\in
R_s^+}(1-e^{-\alpha})\prod\limits_{\alpha\in R_l^+}(1-e^{-\alpha}).
\end{eqnarray*}
\begin{eqnarray}
    \chi\;A_\rho&=&e^{\rho_s+\rho}\prod\limits_{\alpha\in R_s^+}(1-e^{-2\alpha})\prod\limits_{\alpha\in R_l^+}(1-e^{-\alpha})\nonumber\\
    &=&e^{\rho_s+\rho}\prod\limits_{\alpha\in\;(2R_s^+\cup R_l^+)}(1-e^{-\alpha}).\label{XArho}
\end{eqnarray}
If for $\g$,
$(\;\left\|\theta\right\|^2/\left\|\theta_s\right\|^2)=2$ (as in
Case 2) then $(2R_s\cup R_l)$ forms the root system of the dual
algebra, denoted by $\gt$, of $\g$. Then the half sum of positive
roots $\rhot$ of $\gt$ is given by $\rhot=\rho_s+\rho$. Thus by Weyl
denominator identity for $\gt$ and equation (\ref{XArho}):
$$\chi\;A_\rho=\tilde{A}_\rhot=\tilde{A}_{\rho_s+\rho}.$$
Since, $\g$ and $\gt$ have the same Weyl group and $A_\mu$ is the
anti-symmetrizer of $\mu$ w.r.t.~Weyl group,
$\tilde{A}_{\rho_s+\rho}=A_{\rho_s+\rho}.$
\begin{eqnarray*}
    \Rightarrow\;\;\;\chi\;A_\rho&=&A_{\rho_s+\rho}.\\
    \Rightarrow\;\;\;\;\;\;\;\;\chi&=&\frac{A_{\rho_s+\rho}}{A_\rho}=\Char V(\rho_s).
\end{eqnarray*}
\end{proof}

\begin{proof}[Proof for Case 3]
For $n\geq 2$, Panyushev \cite[Prop. 3.8]{P} showed that the set of
all nonzero weights of $V(2\theta_s)$ is $S=2R_s\cup R_s\cup R_l$
with multiplicity of each nonzero weight as 1. For $n=1$, define
$R_s:=R,\;R_s^+:=R^+$ and
$R_l:=\left\{\right\},R_l^+:=\left\{\right\}$. Also define
$\prod_{\alpha\in\left\{\right\}}f(\alpha):=1$ for any function $f$.
Let $\chi:=\Char\Spin_0 V(2\theta_s)$. Thus by Proposition \ref{p2}:
\begin{eqnarray*}
    \chi&=&e^{2\rho_s+\rho}\prod\limits_{\alpha\in R_s^+}(1+e^{-2\alpha})\prod\limits_{\alpha\in R_s^+}(1+e^{-\alpha}) \prod\limits_{\alpha\in
    R_l^+}(1+e^{-\alpha}),\\
    A_\rho&=&e^{\rho}\prod\limits_{\alpha\in
    R_s^+}(1-e^{-\alpha})\prod\limits_{\alpha\in
    R_l^+}(1-e^{-\alpha}).
\end{eqnarray*}

\noindent Therefore
\begin{eqnarray*}
    \chi\;A_\rho&=&e^{2\rho_s+2\rho}\prod\limits_{\alpha\in R_s^+}(1-e^{-4\alpha})\prod\limits_{\alpha\in R_l^+}(1-e^{-2\alpha})\\
&&\;\;(=A_{4\rho}=A_{2(\rho_s+\rho)}\text{ as }\rho_s=\rho\;\text{ for }\;n=1)\\
    &=&e^{2(\rho_s+\rho)}\prod\limits_{\alpha\in\;(2R_s^+\cup R_l^+)}(1-e^{-2\alpha})\;\;\;(\text{ For }n\geq 2)\\
&=&\tilde{A}_{2\rhot}\;\;\;\;\;\;\;\;\;(\text{By comparing with } \tilde{A}_{\rhot}\; \text{of dual algebra}\; \gt \text{ for } n\geq 2)\\
&=&A_{2(\rho_s+\rho)}\;\;(\text{As } \rhot=\rho_s+\rho \text{ as in case 2)}\\
\Rightarrow\;\;\;\chi&=&\frac{A_{(2\rho_s
+\rho)+\rho}}{A_\rho}=\Char V(2\rho_s+\rho)\;\;\;(\text{For }n\geq
1).
\end{eqnarray*}
\end{proof}

\subsubsection{Proof of Proposition \ref{p10}.}
\begin{proof}[Proof of ($\Rightarrow$)]

\noindent Let $\Spin_0(\Vh)$ be irreducible and denote its character by $\chi$. Also let $S$ denote the set of all non-zero weights of $V$ and $q:=e^\delta$. Then by Proposition \ref{p5}:
$$\chi=\;\Char\Spin_0(V)\;\; e^{c\Lambda_0}\prod\limits_{k>0,\beta\in S}(1+e^{-\beta}q^{-k})^{m_{\beta}}
                                        \;\prod\limits_{k>0}(1+q^{-k})^{m_0}$$
where $c=\sum_{\beta\in S^+}\frac{1}{2}m_\beta\beta(\theta^\vee)^2$.
Suppose that $\Char\Spin_0(V)=\sum_{i=1}^s\chi_{\nu_i}$ where
$\chi_{\nu_i}$ is the irreducible character with highest weight\;
$\nu_i$.\;We first show that $s=1$ meaning $V$ is coprimary. Weyl
denominator identity for $\gh$ is:
$$\Ah_{\rhoh}=e^{\rho+h^{\vee}\Lambda_0}\;\prod\limits_{\alpha\in R^+}(1-e^{-\alpha})\;\prod\limits_{k>0,\alpha\in R}(1-e^{-\alpha-k\delta})\;\;\prod\limits_{k>0}(1-e^{-k\delta})^n$$
where $h^{\vee}$ is the dual Coxeter number. Using Weyl denominator identity for $\g$, namely, $A_\rho=e^\rho\prod_{\alpha\in R^+}(1-e^{-\alpha})$ and writing $q=e^\delta$, we can say:
$$\Ah_{\rhoh}=A_{\rho}\;\;e^{h^{\vee}\Lambda_0}\;\prod\limits_{k>0,\alpha\in R}(1-e^{-\alpha}q^{-k})\;\;\prod\limits_{k>0}(1-q^{-k})^n$$
Multiplying $\chi$ with $\Ah_{\rhoh}$, we get:
$$\begin{array}{ccl}
    \chi\cdot\Ah_{\rhoh}&=&\left(\sum\limits_{i=1}^s\chi_{\nu_i}\;A_{\rho}\right)\;e^{(c+h^\vee)\Lambda_0}\;\;\; +\;\;\; (...)q^{-1}\;\;\;+\;\;\;(...)q^{-2}\;\;\;+\;\;\;\ldots\\
                      &=&\left(\sum\limits_{i=1}^s A_{\nu_i+\rho}\right)\;\;e^{(c+h^\vee)\Lambda_0}\;\;\; +\;\;\; (...)q^{-1}\;\;\;+\;\;\;(...)q^{-2}\;\;\;+\;\;\;\ldots\;\;
\end{array}$$
$\chi\cdot\Ah_\rhoh$ will contain the term\;$e^{\nu_i+\rho+(c+h^{\vee})\Lambda_0}=e^{\nu_i+c\Lambda_0+\rhoh}$ for each $i=1\ldots s$ where $\nu_i$ is a dominant weight of $\g$. By character of $\Spin_0(V)$, all its weights of are of the form: $\frac{1}{2}\sum_{\beta\in S^+}a_\beta\beta$ \;for some $-m_\beta\leq a_\beta\leq m_\beta$. Since $c=\frac{1}{2}\sum_{\beta\in S^+}m_\beta\beta(\theta^\vee)^2$, \;$\nu_i(\theta^\vee)\leq c$. Also, since $\nu_i$ is a dominant weight of $\g$, this shows that $\nu_i+c\Lambda_0$ is a dominant weight of $\gh$ for all $i=1\ldots s$. Hence $\chi$ contains irreducible $\chi_{\nu_i+c\Lambda_0}$ for each $i$ in its decomposition into irreducibles. So, $s$ must be $1$ because $\chi=\Char\Spin_0(\Vh)$ is irreducible.

Next, we show that when $V=V(2\theta_s)$ for $\g=\mathfrak{so}_{2n+1}\C$ then $\Spin_0(\Vh)$ is not irreducible.
By \cite[Prop 3.8]{P}, the set of all nonzero weights of $V(2\theta_s)$ is $S=2R_s\cup R_s\cup R_l$ with multiplicity of each nonzero weight as 1. This holds for $n=1$ also, if we define $R_s:=R,\;R_l:=\left\{\right\}$ and $R_s^+:=R^+,\;R_l^+:=\left\{\right\}$. Let $\prod_{\alpha\in\left\{\right\}}f(\alpha):=1$ for any function $f$ and $S^+=2R_s^+\cup R_s^+\cup R_l^+$. Now by Proposition \ref{p5}:
$$
\Char\Spin_0(\Vh)=e^{2\rho_s+\rho+c\Lambda_0}\prod_{\beta\in S^+}(1+e^{-\beta})\prod\limits_{k>0,\beta\in S}(1+e^{-\beta-k\delta})\prod\limits_{k>0}(1+e^{-k\delta})^{m_0}\;.
$$

We calculate the level $c$ of the representation $\Spin_0(\Vh)$ as
follows: For $n\geq 2$, $R_s^+=\left\{L_i\right\}$,
$R_l^+=\left\{L_i\pm L_j:i<j\right\}$, and the dual positive roots
are $\left\{2H_i\right\}\cup\left\{H_i\pm H_j:i<j\right\}$, where
$\left\{H_i:1\leq i\leq n\right\}$ is the dual basis of
$\left\{L_i:1\leq i\leq n\right\}$. Then $\theta=L_1+L_2$ and
\;\;$\theta^\vee=H_1+H_2=H_{\alpha_1}+H_{\alpha_n}+2\sum_{i=2}^{n-1}H_{\alpha_i}$.
Thus:
\begin{eqnarray}
    c&=&\left\{\begin{array}{l}
            2n+3\;\;\;\text{ for }n\geq 2 \\
            10\;\;\;\;\;\;\;\;\;\;\text{ for }n=1.
            \end{array}\right.\nonumber
\end{eqnarray}

Also, the multiplicity of zero weight space, $m_0=n$, the rank of $\g$, by Panyushev \cite[Prop 3.8]{P}. Thus, $\Char\Spin_0(\Vh)$ reduces to:
$$
\chi:=\Char\Spin_0(\Vh)=e^{2\rho_s+\rho}\;\;e^{c\Lambda_0}\prod\limits_{\beta\in S^+}(1+e^{-\beta})\prod\limits_{k>0,\beta\in S}(1+e^{-\beta-k\delta})\prod\limits_{k>0}(1+e^{-k\delta})^n.
$$
The highest dominant weight appearing in $\chi$ is \;$\Lambda:=2\rho_s+\rho+c\Lambda_0$. To prove that $\chi$ is not an irreducible character of $\gh$, it is enough to produce another dominant weight appearing in $\chi$, say,\;$\Lambda^{'}$ such that $(\Lambda-\Lambda^{'})$ can not be expressed as non-negative integral linear combination of simple positive roots of $\gh$.
For $n\geq 2$, take $\Lambda^{'}=2\rho_s+\rho+2L_1-\delta+c\Lambda_0$ which corresponds to the term in $\chi$, $e^{2\rho_s+\rho} \;e^{c\Lambda_0}\; e^{-\beta-k\delta}$ for $\beta=-2L_1\in S$ and $k=1$. Clearly, $\Lambda^{'}$ is dominant as $\lambda^{'}:=2\rho_s+\rho+2L_1$ is dominant weight of $\g$ and $\lambda^{'}(\theta^\vee)=2n+2\;\leq c=2n+3.$
$\Lambda-\Lambda^{'}=-2L_1+\delta=\delta-(L_1+L_2)-(L_1-L_2)=\alpha_0-\alpha_1.$
For $n=1$, take $\Lambda^{'}=2\rho_s+\rho+(2\alpha-\delta)+c\Lambda_0=7\rho-\delta+10\Lambda_0$. Then $\lambda^{'}(\theta^\vee)=7\leq 10$. Here, $\Lambda=3\rho+10\Lambda_0$. So, $\Lambda-\Lambda^{'}=-4\rho+\delta=-2\alpha+\delta=(\delta-\alpha)-\alpha=\alpha_0-\alpha_1$.
\end{proof}
\begin{proof}[Proof of ($\Leftarrow$)]
Using Proposition \ref{p1}(7), it is enough to prove :  $$V=V(\theta_s)\;\;\Rightarrow\;\;\Spin_0(\Vh)=V(\rhoh_s),$$
where, $\rhoh_s:=\rho_s+h_s^\vee\Lambda_0$, $h_s^\vee:=\sum_ia_i^\vee$, $i$'s corresponding to short simple roots of $\gh$.

In the proof of Case 2 of Proposition \ref{p9}, we dealt with finite dimensional Lie algebras $\g$ and its dual algebra $\gt$. In the same spirit, here we will deal with the affine Lie algebra $\gh$ for $\g\in\left\{\mathfrak{so}_{2n+1}\C,\; \mathfrak{sp}_{2n}\C,\;\mathfrak{f}_4\right\}$ and its Langlands dual (obtained by reversing the arrows of the Dynkin diagram of $\gh$) denoted by $\gb$. So if $R$ denotes an object associated to $\g$ (say $R=$ the set of roots of $\g$) then the corresponding object (the (multi-)set of roots) associated to $\gt,\;\gh$ or $\gb$ will be denoted by $\tilde{R},\;\hat{R}$ and $\breve{R}$ respectively.
Let $\chi=\Char\Spin_0(\Vh)$.
In order to completly adapt the proof for Case 2 of Proposition \ref{p9}, we will define multisets associated to roots of $\gh$, namely $\widehat{R}_s^+$ and $\widehat{R}_l^+$ and show the following Facts:
\begin{enumerate}
\item $\chi=e^{\rhoh_s}\prod\limits_{\alpha\in\widehat{R}_s^+}(1+e^{-\alpha})$.
\item $\widehat{R}_s^+\;\cup\;\widehat{R}_l^+=\widehat{R}^+$ the multiset of all positive roots of $\gh$.
\item $2\widehat{R}_s^+\;\cup\;\hat{R}_l^+=\Rb^+$ the multiset of all positive roots of $\gb$.
\item $\rhob=\rhoh_s+\rhoh$.
\end{enumerate}

First we introduce the following notation: For any set $A$, define $A_{\left\{k\right\}}:=$\;a multiset consisting of elements of $A$ with each element appearing $k$ times. Further, the multiset $A_{\left\{1\right\}}$ will be written as $A$.

\begin{proof}[Proof of Fact $2$]
Now, the multiset of all positive roots of $\gh$:
$$\widehat{R}^+:=R^+\;\cup\;\;\left\{\alpha+k\delta:\alpha\in R, k\in\Z_{>0}\right\}\;\;\cup\;\;\left\{k\delta:k\in\Z_{>0}\right\}_{n}$$
For $\g\in\left\{\mathfrak{so}_{2n+1}\C,\; \mathfrak{sp}_{2n}\C,\;\mathfrak{f}_4\right\}$, define:
$$\widehat{R}_s^+:=R_s^+\;\cup\;\;\left\{\alpha+k\delta:\alpha\in R_s, k\in\Z_{>0}\right\}\;\;\cup\;\;\left\{k\delta:k\in\Z_{>0}\right\}_{n_s}$$
where $n_s:=$ number of short simple positive roots of $\g$. Similarly,
$$\widehat{R}_l^+:=R_l^+\;\cup\;\;\left\{\alpha+k\delta:\alpha\in R_l, k\in\Z_{>0}\right\}\;\;\cup\;\;\left\{k\delta:k\in\Z_{>0}\right\}_{n-n_s}.$$
Clearly, $2$ is true.
\end{proof}

\begin{proof}[Proof of Fact $3$]
$$\g\in\left\{\mathfrak{so}_{2n+1}\C,\; \mathfrak{sp}_{2n}\C,\;\mathfrak{f}_4\right\}\;\Rightarrow\;\left(\left\|\theta\right\|^2/\left\|\theta_s\right\|^2\right)=2\; \;\text{in}\;\g $$
$$\Rightarrow\;\;\gh\in\left\{B^{(1)}_{n},C^{(1)}_n,F^{(1)}_4\right\}\;\Rightarrow\;\gb\in\left\{A^{(2)}_{2n-1}, D^{(2)}_{n+1},E^{(2)}_6\right\}.$$
Note that, $\gb$ is a {\it twisted} affine Lie algebra which contains the finite dimensional dual algebra $\gt$ of $\g$. The set of all positive roots of $\gt$ is $2R_s^+\cup R_l^+$. Now 3 is obvious for real roots or one can check directly using \cite[Prop. 6.3]{K}. We only need to check the multiplicities of imaginary positive roots. We observe that, in $2\widehat{R}_s^+\;\cup\;\hat{R}_l^+$ the multiplicity of $2k\delta$ is $n_s+(n-n_s)=n$ and multiplicity of $(2k+1)\delta$ is $n-n_s$ which matches with multiplicities given by V. Kac \cite[Corollary 8.3]{K} for twisted affine Lie algebra $\gb$.
\end{proof}

\begin{proof}[Proof of Fact 4]
Let $\prod_s,\;\prod_l$ denote the set of all simple positive short and long roots of $\gh$ respectively. Also, let $\sum_s,\;\sum_l$ denote the set of all simple positive coroots corresponding to short and long roots of $\gh$ respectively. Then $(2\prod_s\cup\prod_l)$ and $\left(\frac{1}{2}\sum_s\cup\sum_l\right)$ forms the sets of simple positive roots and coroots of $\gb$ respectively. From this we can conclude that $\rhob=\rhoh_s+\rhoh$.
\end{proof}

\begin{proof}[Proof of Fact 1]
Refer to Proposition \ref{p5} for expression for
$\chi:=\Char\Spin_0(\Vh)$. It's easy to check that set of all
nonzero weights of $V=V(\theta_s)$ is $S=R_s$. Since multiplicity of
zero weight space is $n_s$,  Proposition \ref{p5} leads to :
$$\chi=e^{\nu+c\Lambda_0}\prod\limits_{\alpha\in \widehat{R}_s^+}(1+e^{-\alpha})$$
where, $\nu:=\frac{1}{2}\sum_{\alpha\in R_s^+}\alpha=\rho_s$ and $c=\frac{1}{2}\sum_{\alpha\in R_s^+}\alpha(\theta^\vee)^2.$

We will show that $\nu+c\Lambda_0=\rhoh_s$. Since, $\theta$ is always a long root for $\g\in\left\{\mathfrak{so}_{2n+1}\C,\; \mathfrak{sp}_{2n}\C,\;\mathfrak{f}_4\right\}$, $\alpha(\theta^\vee)$ is $0$ or $1$ for all $\alpha\in R_s^+$ due to following Lemma.

\begin{lemma} \label{l3}
{ $$\alpha\in R^+\setminus\left\{\theta\right\}\;\Rightarrow\;\alpha(\theta^\vee)=0\text{ or }1\;.$$}
\end{lemma}
\begin{proof} Verify the following facts about any finite root system $R=R^+\cup R^-$.
\begin{enumerate}
\item $ \alpha\in R^+\;\;\Rightarrow\;\;\alpha+\theta\notin R$\;.
\item $ \alpha\in R^-\;\;\Rightarrow\;\;\alpha-\theta\notin R$\;.
\item For $\alpha\in R^+$,\;
    $\alpha-\theta\in R\;\;\Rightarrow\;\;\alpha-\theta\in R^-\;\;\Rightarrow\;\;\alpha-2\theta\notin R$\;.
\end{enumerate}
Consider the restriction of the adjoint representation, $V(\theta)$,
to the $\mathfrak{sl}_2\C$ corresponding to $\theta$,
$\sm_\theta:=\C X_{\theta}\oplus\C X_{-\theta}\oplus\C \theta^\vee$.
Using above facts, we conclude; for a fixed $\alpha\in
R^+\setminus\left\{\theta\right\}:$
\begin{itemize}
\item   If $\alpha-\theta\notin R$, then $\C X_\alpha$ is a trivial irreducible component of $V(\theta)$ as an $\sm_\theta$-representation, and thus $\alpha(\theta^\vee)=0$.
\item If $\alpha-\theta\in R$, then $\C X_\alpha\oplus\C X_{\alpha-\theta}$ is an irreducible component of $V(\theta)$ as an $\sm_\theta$-representation, and thus $\alpha(\theta^\vee)=1$.
\end{itemize}
This proves Lemma \ref{l3}.
\end{proof}

Thus, $\alpha(\theta^\vee)$ is $0$ or $1$ for all $\alpha\in R_s^+$ $\Rightarrow\;\alpha(\theta^\vee)^2=\alpha(\theta^\vee)$. Therefore, $c=\rho_s(\theta^\vee)=h_s^\vee$ as $\rho_s=$ the sum of all fundamental weights of $\g$ corresponding to simple short roots and $\theta^\vee=\sum_{i=1}^n a_i^\vee\alpha_i^\vee$. We get, $\nu+c\Lambda_0=\rho_s+h_s^\vee\Lambda_0=\rhoh_s$ which proves the Fact 1.
\end{proof}

Finally we get the proof of ($\Leftarrow$) in Proposition \ref{p10} by replacing $\rho,\rho_s,R_s^+$ and $R_l^+$  by $\rhoh,\rhoh_s,\widehat{R}_s^+$ and $\widehat{R}_l^+$
respectively in Proof of Case 2 in Proposition \ref{p9}.
\end{proof}
  
\end{document}